\documentclass[review,onefignum,onetabnum]{siamart171218}
\usepackage{amssymb}
\usepackage{titlesec}
\usepackage{amsmath,amscd}
\usepackage{amsfonts}
\usepackage{graphicx}
\usepackage{ulem}
\usepackage{subfigure}
\usepackage{amssymb}
\usepackage{float} 
\usepackage{multirow}
\usepackage{diagbox}
\usepackage{mathrsfs}
\usepackage{indentfirst}
\usepackage{geometry}
\usepackage{xcolor}
\usepackage{amssymb}
\usepackage{enumitem}
\usepackage{tikz}
\usepackage{etoolbox}
\usepackage{geometry}
\usepackage{fancyhdr}                                
\usepackage{lastpage}    
\usepackage{pifont}                                   
\usepackage{layout}
\usepackage{titlesec}
\usepackage[justification=centering]{caption}
\newtheorem{assumption}{Assumption}[section]

\newcommand{\circled}[2][]{\tikz[baseline=(char.base)]
    {\node[shape = circle, draw, inner sep = 1pt]
    (char) {\phantom{\ifblank{#1}{#2}{#1}}};%
    \node at (char.center) {\makebox[0pt][c]{#2}};}}
\robustify{\circled}

\usepackage{lipsum}
\usepackage{amsfonts}
\usepackage{graphicx}
\usepackage{epstopdf}
\usepackage{algorithmic}
\ifpdf
  \DeclareGraphicsExtensions{.eps,.pdf,.png,.jpg}
\else
  \DeclareGraphicsExtensions{.eps}
\fi

\newsiamremark{remark}{Remark}
\newsiamremark{hypothesis}{Hypothesis}
\crefname{hypothesis}{Hypothesis}{Hypotheses}
\newsiamthm{claim}{Claim}

\headers{Federated primal dual fixed point algorithm}{Yanan Zhu and Jingwei Liang and~ Xiaoqun Zhang}

\title{Federated primal dual fixed point algorithm\thanks{Submitted to the editors DATE.}}

\author{Ya-Nan Zhu and Jingwei Liang and Xiaoqun Zhang}

\usepackage{amsopn}

\ifpdf
\hypersetup{
  pdftitle={Federated primal dual algorithm},
  pdfauthor={Ya-Nan Zhu}
}
\fi

\begin{document}

\maketitle

\begin{abstract}
Federated learning (FL) is a distributed learning paradigm that allows several clients to learn a global model without sharing their private data. In this paper, we generalize a primal dual fixed point (PDFP) \cite{PDFP} method to federated learning setting and propose an algorithm called Federated PDFP (FPDFP) for solving composite optimization problems. In addition, a quantization scheme is applied to reduce the communication overhead during the learning process.
An $O(\frac{1}{k})$ convergence rate (where $k$ is the communication round) of the proposed FPDFP is provided. Numerical experiments, including graph-guided logistic regression, 3D Computed Tomography (CT) reconstruction  are considered to evaluate the proposed algorithm. 
\end{abstract}

\begin{keywords} 
Federated learning, primal-dual fixed point method, quantization.
\end{keywords}

\section{Introduction}
With the availability of massive data and the development of computing capability, machine learning, especially deep learning, has demonstrated unprecedented performance in the past decade. In real-world applications, data is usually disseminated across various companies and portable devices. To utilize the data from different sources and learn a global model, an approach is to share the data with a server with enough storage and computing capability and then use suitable methods to get the resulting model parameters. However, sharing data usually is impossible in practice, and it is also restricted by legal and ethical issues. Thus, it is essential to design learning paradigms that allow different sites/clients to learn a global model without sharing their private data. Federated learning (FL) \cite{FedAvg,FL1} is a novel learning framework that fulfills this goal. The framework consists of a parameter server and a number of clients. The clients collaborate with the server to learn a joint model with their local data and computational resources. To be more precise, suppose we are optimizing the following problem
\begin{equation}\label{pb0}
\underset{x \in \mathbb{R}^d}{\min}~f(x),
\end{equation}
where $f$ is a smooth function in the following form
\begin{equation}
 f(x) = \frac{1}{N}\sum_{i = 1}^N f^{(i)}(x),
~~~\textrm{with}~~~
 f^{(i)}(x) = \frac{1}{n_i}\sum_{j = 1}^{n_i} f_j^{(i)}(x) ,
\end{equation}
where $N$ denotes the number of clients, and $n_i$ is the number of samples of the client $i$. The $f_j^{(i)}$ is the $j$'th loss function of client $i$. 
To solve the above problem in a distributed way, a representative algorithm, FedAvg \cite{FedAvg}, takes the following iterative strategy (see also Algorithm, {\ref{FedAvg}}. The superscript $(i)$ and subscript $k$ denote the index of clients and communication round, respectively.)
\begin{enumerate}[label={\arabic{*}).}]
  \item In round $k$, the server selects $0 < n \leq N$ clients uniformly at random and sends them the global parameter $x_k$.
  \item Each of the selected clients $i$ runs $\tau > 0$ steps mini-batch stochastic gradient descent (SGD)~\cite{SGD} (e.g. $\tau=10$ in Algorithm {\ref{FedAvg}}) to get the local model $x_{k + 1,\tau}^{(i)}$, and then uploads the $x_{k + 1,\tau}^{(i)}$ to the server.
  \item The server aggregates the local updates and gets the global model $x_{k + 1}$ for the next round. 
\end{enumerate}
\begin{algorithm}
\caption{\textbf{Federated Averaging \cite{FedAvg}}}
\label{FedAvg}
\begin{algorithmic}[1]
\REQUIRE Number of clients $N$, batch size $b$ for local SGD updates, number of total communication $K$, and the participation number $n$ in each round. Properly chosen step-size $\gamma_k > 0$. \\
\STATE    \textbf{For} $k = 0,1,2,\cdots,K - 1$ \\
\STATE    \qquad\quad Server selects $n$ clients $\mathcal{S}_k \subset \{1,2,\cdots,N\}$ uniformly at random. \\
\STATE    \qquad\quad Server sends global model parameter $x_k$ to the selected clients.  \\
\STATE    \qquad\quad \textbf{For} clients $i$ in $\mathcal{S}_k$ do  \\
\STATE    \qquad\quad  ~~~~ $x_{k,0}^{(i)} = x_k$\\
\STATE    \qquad\quad  ~~~~ \textbf{For} $t = 0,1,2,\cdots,\tau - 1$ \\
\STATE    \qquad\quad\qquad   ~~~~ Choose a subset $I_i \subset \{1,\cdots,n_i\} $ with size $b$ uniformly at random.
\STATE    \qquad\quad\qquad   ~~~~ $\tilde{\nabla} f^{(i)}(x_{k,t}^{(i)}) = \frac{1}{b} \sum_{j \in I_i} \nabla f_{j}^{(i)}(x_{k,t}^{(i)})$ \\
\STATE    \qquad\quad\qquad   ~~~~ $x_{k + 1,t + 1}^{(i)} = x_{k} - \gamma_k \tilde{\nabla} f^{(i)}(x_{k,t}^{(i)})$\\
\STATE    \qquad\quad  ~~~~ \textbf{End For} 
\STATE    \qquad\quad  ~~~~ Send $x_{k + 1,\tau}^{(i)}$ to the server
\STATE    \qquad\quad \textbf{End For} \\
\STATE    \qquad\quad $x_{k + 1} = \frac{1}{n}\sum_{i \in \mathcal{S}_k} x_{k + 1,\tau}^{(i)}$.\\
\STATE  \textbf{End For} \\
\ENSURE $x_{K}$.
\end{algorithmic}
\end{algorithm}
It should be remarked that only a subset of clients is activated in each round. One of the reasons causing this is that, in FL systems, clients are not always able to communicate with the server. For example, when a mobile phone is powered off or loses internet connection, it cannot upload data to the server.  
After FedAvg, several modifications from algorithmic \cite{FedPD,Fedpaq,FCM,Fedprox,Smith} and privacy \cite{PSG,HE,pDLG,Backdoor,DP,GANlk} perspectives have been proposed. 
For example, the FedProx \cite{Fedprox} adds a proximal term $\lVert x - x_k\rVert_2^2$ to the local objectives of FedAvg and solves it inexactly using SGD. The additional proximal term ensures that the local updated model $x_k$ is close to the global model of the previous round, which helps to address the heterogeneity across the clients. 

In each communication round, FedAvg or its variants need to communicate model parameters to the server. In practice, a large number of clients and limited communication bandwidth prevent enough clients from participating in the averaging step, which impedes the learning process from generating a satisfactory solution. One approach to address the communication constraints is applying compression or quantization techniques where each entry of the model parameters is rounded to a set of discrete values so that fewer bits are required to represent these values. There are a series of works combing quantization techniques with first-order stochastic optimization methods, such as \cite{QSGD,DCGD,NUQSGD,DIANA,PowerSGD,signSGD,IntSGD,ADIANA,MARINA}, to mention a few. In this paper, we focus on the low precision quantization \cite{QSGD} which is defined as the follows
\begin{definition}\label{def1}
For any $x \not = \mathbf{0} \in \mathbb{R}^d$, the low precision
quantizer $Q_s: \mathbb{R}^d \to \mathbb{R}^d$ is defined as
\begin{equation}
Q_s(x_i) = \lVert x \rVert \cdot \mathrm{sign}(x_i)\cdot \xi_i(x,s),~~~ i \in 1,2,\cdots,d, 
\end{equation}
where $\xi_i(x_i,s)$ is a random variable defined by
\begin{equation}
\xi_i(x_i,s) = \left \{ 
\begin{array}{ll}
& \frac{\ell}{s}, \qquad \mathrm{with ~ probability ~ } 1 - \frac{|x_i|}{\lVert x \rVert}s + \ell \\
& \frac{\ell + 1}{s}\qquad  otherwise.
\end{array}
\right.
\end{equation}
and $\ell$ is an integer such that $\frac{|v_i|}{\lVert v \rVert} \in [\frac{\ell}{s}, \frac{\ell + 1}{s}]$.
\end{definition}
Note that for a given vector $x$, instead of storing $d$ float number, the low precision representation $Q_s(x)$ only needs to store a float number $\lVert x \rVert$ and $d$ integers. The communication overhead can be further reduced when coding strategies, such as \textit{Elias} integer coding \cite{Elias}, are used. 
Also, it can be verified that  $Q_s(x)$ is an unbiased estimator of $x$, i.e.,
\begin{equation}\label{ep}
 \mathbb{E}[Q_s(x)] = x,
\end{equation}
where the expectation is with respect to the random variable $\xi_i(x_i,s)$.  The variance between $Q_s(x)$ and $x$ is bounded by norm squared of $x$ multiplied by some positive constant $\alpha$ related to dimension $d$ and quantization level $s$ \cite{QSGD}, i.e.,
\begin{equation}\label{vr}
 \mathbb{E}[\lVert Q_s(x) - x \rVert_2^2] \leq \alpha \lVert x \rVert_2^2.
\end{equation}
Properties {(\ref{ep})} and {(\ref{vr})} of quantization $Q_s(\cdot)$ are essential in the convergence analysis of many stochastic first-order optimization algorithms. Especially in federated learning, \cite{Fedpaq} combines the FedAvg with quantization $Q_s(\cdot)$ and develops an algorithm called Federated Periodic Averaging and Quantization (FedPAQ) which, in each round, clients send the quantized parameters to the server to reduce communication overhead. The detailed updating rule of FedPAQ can be found in Algorithm {\ref{Fedpaq}}. 

\begin{algorithm}
\caption{\textbf{Federated Periodic Averaging and Quantization}  \cite{Fedpaq}}
\label{Fedpaq}
\begin{algorithmic}[1]
\REQUIRE Input the number of clients $N$, batch size $b$ for local stochastic gradient, the number of communication $K$, and the participation number $n$ in each round and choose proper step size $\gamma_k > 0$ and quantization level $s$. \\
\STATE    \textbf{For} $k = 0,1,2,\cdots,K - 1$ \\
\STATE    \qquad\quad Server select $n$ clients $\mathcal{S}_k \subset \{1,2,\cdots,N\}$ uniformly at random.  \\
\STATE    \qquad\quad Server sends averaged model parameter $x_k$ to the selected clients.  \\
\STATE    \qquad\quad \textbf{For} clients $i$ in $\mathcal{S}_k$ do  \\
\STATE    \qquad\quad  ~~~~ $x_{k,0}^{(i)} = x_k$\\
\STATE    \qquad\quad  ~~~~ \textbf{For} $t = 0,1,2,\cdots,\tau - 1$ \\
\STATE    \qquad\quad\qquad   ~~~~ Choose a subset $I_i \subset \{1,\cdots,n_i\} $ with size $b$ uniformly at random.
\STATE    \qquad\quad\qquad   ~~~~ $\tilde{\nabla} f^{(i)}(x_{k,t}^{(i)}) = \frac{1}{b} \sum_{j \in I_i} \nabla f_{j}^{(i)}(x_{k,t}^{(i)})$ \\
\STATE    \qquad\quad\qquad   ~~~~ $x_{k + 1,t + 1}^{(i)} = x_{k} - \gamma_k \tilde{\nabla} f^{(i)}(x_{k,t}^{(i)})$\\
\STATE    \qquad\quad  ~~~~ \textbf{End For} 
\STATE    \qquad\quad   ~~~~ Send $Q_s(x_{k + 1,\tau}^{(i)} - x_k)$ to the server
\STATE    \qquad\quad \textbf{End For} \\
\STATE    \qquad\quad $x_{k + 1} = x_k + \frac{1}{n}\sum_{i \in \mathcal{S}_k} Q_s(x_{k + 1,t + 1}^{(i)} - x_k)$\\
\STATE  \textbf{End For} \\
\ENSURE $x_{K}$.
\end{algorithmic}
\end{algorithm}

In practice, there are also problems involving a non-smooth term in the objective, i.e.,
\begin{equation}\label{pb01}
\underset{x \in \mathbb{R}^d}{\min}~f(x) + g(x), 
\end{equation}
where $g$ is proper and lower semi-continuous (l.s.c.). The problem (\ref{pb01}) recovers a wide range of applications, for example, LASSO \cite{LASSO}, penalized M estimation \cite{PM}, penalized logistic regression \cite{plogistic}, total variation problems \cite{TV,TVPGD}. In general the proximity operator of $g$ at point $y$ defined by $\mathrm{Prox}_g(y) = \arg\min_{x \in \mathbb{R}^d}~g(x) + \frac{1}{2}\lVert x - y \rVert_2^2$ 
is assumed to be  easy to compute. To handle the non-smoothness, the authors in \cite{FCM} use mirror descent \cite{CMI} to solve {(\ref{pb01})} and propose a method called Federated Mirror Descent (FedMID). Moreover, to mitigate the so-called ``curse of primal averaging", they proposed Federated Dual Averaging (FedDualAvg) method, where the averaging step operates in the dual space. The acceleration of FedDualAvg with statistical recovery guarantee can be found in \cite{baofast}. 

The algorithms mentioned above optimize the objective either with a smooth loss $f(x)$ or a smooth loss plus a non-smooth regularization term whose proximity operator is easy to evaluate. However, there are a large portion of applications in imaging sciences and graph-guided classification problems in data sciences that the regularization term is composed with a matrix, and the problems take the following form
\begin{equation}\label{pb1}
\underset{x \in \mathbb{R}^d}{\min}~f(x) + g (B x),	
\end{equation}
where $f$ is a proper smooth function defined as in {(\ref{pb0})}. 
Here we assume $f$ is smooth convex and has $\frac{1}{\beta}$-Lipschitz continuous gradient for some $\beta > 0$, $g$ is proper convex l.s.c. and may not be differentiable. $B:\mathbb{R}^d \rightarrow \mathbb{R}^{m}$ is a linear transform. Due to the operator $B$, a closed-form expression of the proximity operator of $g \circ B$ in general does not exist, and a direct application of methods such as mirror descent or proximal gradient descent (PGD) \cite{PGD} needs a subroutine to compute it. To avoid sub-problem solving, we resort to the following min-max saddle-point reformulation of the problem {(\ref{pb1})},  
\begin{equation}\label{pb2}
\underset{x \in \mathbb{R}^d}{\min}\underset{v \in \mathbb{R}^m}{\max}~f(x) + \langle Bx,v \rangle - g^*(v), 
\end{equation}
where $g^*(\cdot)$, defined by $g^*(v) = \underset{y \in \mathbb{R}^m}{\sup}\langle v,y \rangle - g(y)$, is the conjugate function of $g$. 

Compared to problem {(\ref{pb1})}, the reformulation {(\ref{pb2})} involves conjugate $g^*$ whose proximity operator is as easy to compute as the proximity operator of $g$. This allows several primal dual splitting algorithms \cite{CP,PDFP,condat,vu,HeYuan,com}.
Here we concentrate on the primal dual fixed point method (PDFP) \cite{PDFP} (see also \cite{PAPC,LV}) whose updating rule is given in Algorithm {\ref{PDFP}}. 

\begin{algorithm}
\caption{\textbf{Primal Dual Fixed Point Method \cite{PDFP}}}
\label{PDFP}
\begin{algorithmic}[1]
\REQUIRE  Choose proper $\gamma > 0,\lambda > 0$
\STATE   \textbf{For} $k = 0,1,2,\cdots,K - 1$ \\
\STATE   \qquad $x_{k + \frac{1}{2}} = x_{k} - \gamma \nabla f(x_{k})$ \\
\STATE   \qquad $v_{k + 1} = \mathrm{Prox}_{\frac{\lambda}{\gamma}g^*}\big(\frac{\lambda}{\gamma}B x_{k + \frac{1}{2}}  + (I - \lambda BB^T)v_{k} \big)$ \\
\STATE   \qquad $x_{k + 1} = x_{k + \frac{1}{2}} - \gamma B^T v_{k + 1}$\\
\STATE   \textbf{End} \\
\ENSURE $x_{K},v_{K}$.
\end{algorithmic}
\end{algorithm}

It can be verified, if $\lambda = 1$ and $B = I$ with $I$ being the identity matrix, PDFP reduces to the projected gradient descent method (PGD, see the next section for a detailed derivation). Thus PDFP can be seen as a generalization of PGD when a general matrix $B$ is considered. 
As a result, the main motivation of this work is to generalize PDFP to federated learning setting and reduce the communication overhead by adopting quantization techniques. To that end, we list our contributions as follows:
\begin{itemize}
   \item To solve the problem where the objective has a complex regularization term, and the data are distributed across several clients, we propose a federated primal dual optimization algorithm (see Algorithm \ref{FPDFP}) which is called ``FPDFP'' for short. In addition, the quantization technique is adopted to ensure communication efficiency.
   \item We provide the convergence and convergence rate of the method under some standard assumptions.
   \item Numerical experiments including graph-guided logistics regression and 3D CT reconstruction are carried out to validate the performance of the proposed method. 
\end{itemize} 

\paragraph{Paper organization}
The paper is organized as followed: the algorithm FPDFP is elaborated in Section \ref{sec:algorithm}. Theoretical analysis is provided in Section \ref{sec:analysis}, while numerical experiments are provided in Section \ref{sec:experiment}. Some useful preliminary results are collected in the Appendix \ref{sec:appendix}. 


\section{Algorithm}\label{sec:algorithm}
In this section, we provide the details of our proposed algorithm FPDFP, see below in Algorithm \ref{FPDFP}.


\begin{algorithm}
\caption{\textbf{Federated Primal Dual Fixed Point Algorithm}}
\label{FPDFP}
\begin{algorithmic}[1]
\REQUIRE Choose proper $\gamma_k > 0,\lambda > 0$, number of clients $N$, initial point for each client $x_0^{(1)}= x_0^{(2)} = \cdots = x_0^{(N)} \in \mathbb{R}^d$ and 
$v_0^{(1)}= v_0^{(2)} = \cdots = v_0^{(N)} \in \mathbb{R}^{r}$, batch size $b$ for local stochastic gradient, participation number $n$ in each round. \\
\STATE    \textbf{For} $k = 0,1,2,\cdots,K - 1$ \\
\STATE    \qquad\quad Server selects $n$ clients $\mathcal{S}_k \subset \{1,2,\cdots,N\}$ uniformly at random.  \\
\STATE    \qquad\quad Server sends averaged model parameter $x_k,v_k$ to the selected clients.  \\
\STATE    \qquad\quad \textbf{For} clients $i$ in $\mathcal{S}_k$ do \\
\STATE    \qquad\quad  ~~~~ Choose a subset $I_i \subset \{1,\cdots,n_i\} $ with size $b$ uniformly at random. \\
\STATE    \qquad\quad  ~~~~ $\tilde{\nabla} f^{(i)}(x_{k}) = \frac{1}{b} \sum_{j \in I_i} \nabla f_{j}^{(i)}(x_{k})$ \\
\STATE    \qquad\quad  ~~~~ $x_{k + \frac{1}{2}}^{(i)} = x_{k} - \gamma_k \tilde{\nabla} f^{(i)}(x_{k})$
\STATE    \qquad\quad  ~~~~ $v_{k + 1}^{(i)} = \mathrm{Prox}_{\frac{\lambda}{\gamma_k}g^*}\big(\frac{\lambda}{\gamma_k}B x_{k + \frac{1}{2}}^{(i)}  + (I - \lambda BB^T)v_{k} \big)$ \\
\STATE    \qquad\quad  ~~~~ $x_{k + 1}^{(i)} = x_{k + \frac{1}{2}}^{(i)} - \gamma_k B^Tv_{k + 1}^{(i)}$\\
\STATE    \qquad\quad  ~~~~ Send $Q(x_{k + 1}^{(i)} - x_k)$ and $Q(v_{k + 1}^{(i)})$ to server\\
\STATE    \qquad\quad \textbf{End For} \\
\STATE    \qquad\quad $x_{k + 1} = x_k + \frac{1}{n}\sum_{i \in \mathcal{S}_k} Q(x_{k + 1}^{(i)} - x_k)$ and $v_{k + 1} = \frac{1}{n}\sum_{i \in \mathcal{S}_k} Q(v_{k + 1}^{(i)})$ \\
\STATE  \textbf{End For} \\
\ENSURE $x_{K},v_{K}$.
\end{algorithmic}
\end{algorithm}

In summary, the proposed algorithm consists of three main steps: 
\begin{itemize}
  \item Step 1: In each round, the server selects $n$ clients uniformly at random and sends global primal and dual variable $x_k, v_k$ to them.
  \item Step 2: The selected clients run one step PDFP and obtain $x_{k + 1}^{(i)},v_{k + 1}^{(i)}$, then send the quantized message $Q(x_{k + 1}^{(i)} - x_k)$ and $Q(v_{k + 1}^{(i)})$ to the server.
  \item Step 3: The server aggregates the received quantized message from the selected clients and computes the global model $x_{k + 1}, v_{k + 1}$ for the next round.
\end{itemize}

\begin{remark}\label{rmk1}
It can be observed that FPDFP needs to transmit two variables $x_k, v_k$ in each round, while existing federated learning algorithms such as FedSGD \cite{FedSGD}, FedAvg \cite{FedAvg}, FedPAQ \cite{Fedpaq},  only needs to handle $x_k$. This is mainly caused by the fact we are primal-dual algorithm, while these algorithm cannot handle the composite term $g(Bx)$. 
Furthermore, when $B = I$ is the identity matrix, if we choose $\lambda = 1$, one does not need to send the dual variable $Q (v_k)$. 
Since for this scenario, FPDFP for client $i$ at round $k$ becomes
\begin{equation}\label{com1}
\begin{aligned}
x_{k + \frac{1}{2}}^{(i)} & = x_{k} - \gamma_k \tilde{\nabla} f^{(i)}(x_{k}); \\
v_{k + 1}^{(i)} & = \mathrm{Prox}_{\frac{\lambda}{\gamma_k}g^*}\big(\tfrac{\lambda}{\gamma_k}x_{k + \frac{1}{2}}^{(i)} \big); \\
x_{k + 1}^{(i)} & = x_{k + \frac{1}{2}}^{(i)} - \gamma_k v_{k + 1}^{(i)}.
\end{aligned}   
\end{equation} 
which can be simplified to 
\begin{equation}\label{com2}
\begin{aligned}
v_{k + 1}^{(i)} & = \mathrm{Prox}_{\frac{\lambda}{\gamma_k}g^*}\big(\tfrac{1}{\gamma_k}\big(x_{k} - \gamma_k \tilde{\nabla} f^{(i)}(x_{k})\big)\big); \\
x_{k + 1}^{(i)} & = x_{k} - \gamma_k \tilde{\nabla} f^{(i)}(x_{k}) - \gamma_k v_{k + 1}^{(i)}.
\end{aligned}   
\end{equation}
Applying  Moreau's identity \cite{cvxbook} leads to the following local  update
\begin{equation}\label{com3}
 x_{k + 1}^{(i)} =  \mathrm{Prox}_{g}\big(x_{k} -{} \gamma_k \tilde{\nabla} f^{(i)}(x_{k})\big),
\end{equation}
which is simply the proximal stochastic gradient descent step \cite{PSGD}. 
\end{remark}

As a consequence, we have the following Algorithm {\ref{FPDFP2}} for dealing with problem with $B = I$.

\begin{algorithm}
\caption{\textbf{Federated Primal Dual Fixed Point Algorithm for $B = I$.}}
\label{FPDFP2}
\begin{algorithmic}[1]
\REQUIRE Choose proper $\gamma_k > 0$, number of clients $N$, initial point for each client $x_0^{(1)}= x_0^{(2)} = \cdots = x_0^{(N)} \in \mathbb{R}^d$, batch size $b$ for local stochastic gradient, participation number $n$ in each round. \\
\STATE    \textbf{For} $k = 0,1,2,\cdots,K - 1$ \\
\STATE    \qquad\quad Server selects $n$ clients $\mathcal{S}_k \subset \{1,2,\cdots,N\}$ uniformly at random.  \\
\STATE    \qquad\quad Server sends averaged model parameter $x_k$ to the selected clients.  \\
\STATE    \qquad\quad \textbf{For} clients $i$ in $\mathcal{S}_k$ do \\
\STATE    \qquad\quad  ~~~~ Choose a subset $I_i \subset \{1,\cdots,n_i\} $ with size $b$ uniformly at random. \\
\STATE    \qquad\quad  ~~~~ $\tilde{\nabla} f^{(i)}(x_{k}) = \frac{1}{b} \sum_{j \in I_i} \nabla f_{j}^{(i)}(x_{k})$ \\
\STATE    \qquad\quad  ~~~~ $x_{k + 1}^{(i)} =  \mathrm{Prox}_{g}\big(x_{k} -{} \gamma_k \tilde{\nabla} f^{(i)}(x_{k})\big)$ \\
\STATE    \qquad\quad  ~~~~ Send $Q(x_{k + 1}^{(i)} - x_k)$ to server.\\
\STATE    \qquad\quad \textbf{End For} \\
\STATE    \qquad\quad $x_{k + 1} = x_k + \frac{1}{n}\sum_{i \in \mathcal{S}_k} Q(x_{k + 1}^{(i)} - x_k)$ \\
\STATE  \textbf{End For} \\
\ENSURE $x_{K}$.
\end{algorithmic}
\end{algorithm}

\begin{remark}\label{rmk2}
The FPDFP is connected with several established works in the literature:
\begin{itemize}
   \item When function $g = 0$ in {(\ref{pb1})}, it can be verified that the dual variable does not participate in the update, and FPDFP becomes one step local update version of FedPAQ \cite{Fedpaq}. As it has been shown in \cite{Fedpaq} that FedPAQ recovers FedAvg when non quantization is used which indicates FPDFP also recovers FedAvg in that scenario. The convergence analysis of the current version of FPDFP only holds for one step local update, we leave the  multi-step extension of FPDFP to the future work. 
   
   \item When $B = I$, just as we discussed in {(\ref{com1}) - (\ref{com3})}, the local update of FPDFP becomes proximal stochastic gradient descent, and the most related work to the FPDFP in this case is the FedMid \cite{FCM}. However, FPDFP allows low precision message transmission so that the communication overhead can be reduced. 
 \end{itemize} 
\end{remark}

\section{Convergence Analysis}\label{sec:analysis}
In this section, we present the convergence analysis of the proposed FPDFP. 
Before proving the main theorem, we list several key assumptions and preliminary lemmas below. 
\begin{assumption}\label{smooth}

The function $f$ is smooth with $\frac{1}{\beta}$ Lipschitz continuous gradient.
\end{assumption}

\begin{assumption}\label{strongcvx}
The function $f(\cdot)$ is $\mu$-strongly convex i.e., for any $x,y \in \mathbb{R}^d$, $\langle \nabla f(x) - \nabla f(y),x - y \rangle \geq \mu \lVert x - y\rVert_2^2$.
\end{assumption}

\begin{assumption}\label{Quantizer}
\cite{Fedpaq} The low precision quantization $Q_s(\cdot)$ under quantization level $s$ is unbiased, and its variance grows with the squared of $\ell_2$-norm of its argument i.e.
\begin{equation}
 \mathbb{E}[Q_s(x) | x] = x, \qquad \mathbb{E}[\lVert Q_s(x) - x \rVert_2^2 | x] \leq q \lVert x \rVert_2^2. 
\end{equation}
\end{assumption}

\begin{assumption}\label{VR1}
The stochastic gradient for each client $i$ is unbiased i.e., $\mathbb{E}[\tilde{\nabla} f^{(i)}(x) ] = \nabla f^{(i)}(x)$  and its variance is uniformly bounded $\mathbb{E}[\lVert \tilde{\nabla} f^{(i)}(x) - \nabla f^{(i)}(x) ] \rVert_2^2 \leq \sigma^2$. 
\end{assumption}

\begin{assumption}\label{VR2}
The variance between the gradient $\nabla f^{(i)}(x)$ of each client $i$ and $\nabla f(x)$ is uniformly bounded i.e.,  $\mathbb{E}[\lVert \nabla f^{(i)}(x) - \nabla f(x) \rVert_2^2 ] \leq \delta^2$. 
\end{assumption}

\begin{assumption}\label{fact1}
The dual variable $v_{k}^{(i)}, i = 1,\cdots,N$ in FPDFP is uniformly bounded by $M$.
\end{assumption}

The Assumption {\ref{smooth}} is a standard assumption across the literature. The Assumption {\ref{strongcvx}} is necessary for us to show the convergence of the FPDFP since our convergence analysis is based on the estimate of the discrepancy between iterates and the optimal point (see Appendix \ref{sec:appendix}).

The Assumption {\ref{Quantizer}} and {\ref{VR2}} show the unbiasedness and variance of quantization operation $Q_s(\cdot)$ and stochastic gradient, which are standard and essential properties in the convergence analysis of stochastic algorithms. 
The last assumption may seem weird at first glance. However, in many sparsity promoting applications where $g$ takes (group) $\ell_1$-norm, the conjugate $g^*$ is the indicator function of a compact set, hence the dual variable $v_{k}^{(i)}$ is computed, and Assumption {\ref{fact1}} is satisfied. 


\begin{lemma}\label{lm1}
Suppose that $x^*$ is a solution of (\ref{pb1}), then there exists $v^*  \in \partial g(Bx^*)$ such that for any $\lambda > 0$ and a sequence $ \{ \gamma_k\}  > 0$ such that 
\begin{equation}\label{lm1eq1}
\left\{
\begin{aligned}
v^* & = \mathrm{Prox}_{\frac{\lambda}{\gamma_k} g^*}\big(\tfrac{\lambda}{\gamma_k} B(x^* - \gamma_k\nabla f(x^*))
+ (I - \lambda BB^T)v^* \big)     \\
& = T_k(x^*,v^*) ;\\
x^* & = x^* - \gamma_k \nabla f(x^*) - \gamma_k B^TT_k(x^*,v^*) . 
\end{aligned}
\right.
\end{equation}
Conversely, if $(x^*, v^*)$ satisfies {(\ref{lm1eq1})}, then $x^*$ is a solution to the problem {(\ref{pb1})}. 
\end{lemma}

Lemma {\ref{lm1}} states the optimality condition of problem {(\ref{pb1})}, whose proof can be found in the Appendix. Based on Lemma {\ref{lm1}}, we present a key lemma below.

\begin{lemma}\label{lm6}
Suppose Assumptions {\ref{smooth}}-{\ref{fact1}} hold, and a decreasing step size  $\gamma_k$ is used in Algorithm \ref{FPDFP}, choose $0 < \lambda \leq \frac{1}{\rho_{\max}(BB^T)}$ and $0 < D_1 < 2\mu$, then there exists a constant $K_0$  such that for $k \geq K_0$, the following estimate holds 
\begin{equation}\label{lm6eq0}
\begin{aligned}
 &\mathbb{E}_{k + 1}[\lVert x_{k + 1} - x^*\rVert_2^2 + \tfrac{\gamma_{k + 1}^2}{\lambda}\lVert v_{k + 1} - v^*\rVert_2^2] \\
& \leq (1 - D_1\gamma_k)\mathbb{E}_{k}[\lVert x_k - x^* \rVert_2^2] + \tfrac{\gamma_{k}^2}{\lambda}(1 -  \lambda\rho_{min}(BB^T)) \mathbb{E}_{k}[\lVert v_k - v^* \rVert_2^2] + \gamma_k^2D_2,\\
 \end{aligned}
\end{equation}
where $x_{k + 1},v_{k + 1}$ are the iterates of $k + 1$ round of Algorithm \ref{FPDFP}. The $\mathbb{E}_{k + 1}[\cdot]$ is expectation up to $k + 1$ round and
\begin{equation}
D_2 = \sigma^2 + \delta^2 
+ \Big(\frac{4(1 + q)(N - n)}{n(N - 1)} + 1 + q\Big)C_0,   
\end{equation}
where
$C_0 = 4\sigma^2 + 8\delta^2 + 2\rho_{\max}(BB^T)(2M^2+ \frac{M^2}{\lambda})$ and $C_0 = 0$ when full participation and no quantization is performed.

\end{lemma}


With Lemma {\ref{lm6}}, we are in position to present the convergence properties of FPDFP.

\begin{theorem}\label{thm1}
Suppose Assumptions {\ref{smooth}}-{\ref{fact1}} hold, and decreasing step size $\gamma_k$ is chosen such that $\sum_{k = 1}^{\infty} \gamma_k = +\infty$ and $\sum_{k = 1}^{\infty} \gamma_k^2 < +\infty$, and choose $0 < \lambda \leq \frac{1}{\rho_{\max}(BB^T)}$, we then have
\begin{equation}\label{thm1eq0}
\lim\inf_{k \to + \infty}~ \mathbb{E}_{k}[\lVert x_k - x^* \rVert_2^2] = 0,
\end{equation} 
where $x_k$ is the iterate of Algorithm \ref{FPDFP} and $x^*$ is an optimal point of the problem {(\ref{pb1})}.
\end{theorem}

\begin{proof}
Recall {(\ref{lm6eq0})}, for all $k \geq K_0$
\begin{equation}\label{thm1eq1}
\begin{aligned}
& \mathbb{E}_{k + 1}[\lVert x_{k + 1} - x^* \rVert_2^2 + \tfrac{\gamma_{k + 1}^2}{\lambda}\lVert v_{k + 1}  - v^* \rVert_2^2] \\
& \leq (1 - D_1\gamma_k)\mathbb{E}_{k}[\lVert x_k - x^* \rVert_2^2] + \tfrac{\gamma_{k}^2}{\lambda}(1 - \rho_{min}(BB^T)) \mathbb{E}_{k}[\lVert v_k - v^* \rVert_2^2] + \gamma_k^2D_2 \\
& \leq \mathbb{E}_{k}[\lVert x_k - x^* \rVert_2^2 + \tfrac{\gamma_{k}^2}{\lambda}\mathbb{E}_{k}(\lVert v_k - v^* \rVert_2^2)] - D_1\gamma_k{\mathbb{E}_{k}[\lVert x_k - x^* \rVert_2^2]} + \gamma_k^2D_2.
\end{aligned}  
\end{equation}
Sum {(\ref{thm1eq1})} from $k = 1$ to $+\infty$, we obtain 
\begin{equation}
\begin{aligned}
D_1\gamma_k\sum_{k = k_0}^{+\infty}{\mathbb{E}_{k}[\lVert x_k - x^* \rVert_2^2} 
& \leq \lVert x_{k_0} - x^* \rVert_2^2 + \frac{\gamma_{k_0}^2}{\lambda} \lVert v_{k_0} - v^* \rVert_2^2 
+ D_2\sum_{k = k_0}^{+\infty}\gamma_k^2.
\end{aligned}
\end{equation}
Since $\sum_{k = k_0}^{+\infty}\gamma_k^2 < +\infty$ and $\sum_{k = k_0}^{+\infty}\gamma_k = +\infty$, we get that 
\begin{equation}
\lim\inf_{k \to + \infty}~\lVert x_k - x^* \rVert_2^2 = 0.
\end{equation}
This completes the proof.
\end{proof}

Before presenting the convergence rate result, we introduce a lemma from \cite[Lemma 5]{Fedpaq}.

\begin{lemma}[{\cite[Lemma 5]{Fedpaq}}]\label{rate}
Let $\Delta_{k}$ be a non-negative sequence satisfying $\Delta_{k + 1} \leq \big(1 - \tfrac{2}{k + c}\big)\Delta_k + \tfrac{a}{(k + c)^2}$, then 
for every $k \geq k_0$, where $a,c$ are positive reals and $k_0$ is a positive integer. Then for any $k \geq k_0$, there holds $\Delta_{k} \leq \tfrac{(k_0 + c)^2}{(k + c)^2}\Delta_{k_0} + \tfrac{a}{k + c}$. 
\end{lemma}

Now we are ready to present convergence rate of the proposed algorithm. 

\begin{theorem}\label{thm2}
Suppose Assumptions {\ref{smooth}}-{\ref{fact1}} hold, and the step size is chosen as $\gamma_k = \frac{2}{D_1k + 1}$ for some $0 < D_1 < 2\mu$, and choose $0 < \lambda \leq \frac{1}{\rho_{\max}(BB^T)}$. Furthermore the matrix $B$ has full row rank, then for $k \geq K$ with some $K > 0$ larger enough, the following estimate holds
\begin{equation}\label{thm2eq0}
\mathbb{E}_{k + 1}[\lVert x_{k + 1} - x^*\rVert_2^2 + \tfrac{\gamma_{k + 1}^2}{\lambda}\lVert v_{k + 1} - v^*\rVert_2^2] 
\leq  \tfrac{(D_1K + 1)^2}{(D_1 k + 1)^2}\mathbb{E}_{K}[\lVert x_{K} - x^* \rVert_2^2 + \tfrac{\gamma_{K}^2}{\lambda}\lVert v_{K} - v^* \rVert_2^2] + \tfrac{4D_2}{(D_1k + 1)D_1},
\end{equation}
where the constant $D_1$ and $D_2$ are defined as that in Lemma {\ref{lm6}}. 
\end{theorem}
\begin{proof}
Again invoke {(\ref{lm6eq0})}, for $k \geq K_0$,
\begin{equation}\label{thm2eq1}
\begin{aligned}
& \mathbb{E}_{k + 1}[\lVert x_{k + 1} - x^* \rVert_2^2 + \tfrac{\gamma_{k + 1}^2}{\lambda}\lVert v_{k + 1}  - v^* \rVert_2^2] \\
& \leq (1 - D_1\gamma_k)\mathbb{E}_{k}[\lVert x_k - x^* \rVert_2^2] + \tfrac{\gamma_{k}^2}{\lambda}(1 - \lambda \rho_{min}(BB^T)) \mathbb{E}_{k}[\lVert v_k - v^* \rVert_2^2] + \gamma_k^2D_2. 
\end{aligned}  
\end{equation}
Since the matrix $B$ is full row rank, then $\rho_{min}(BB^T) > 0$. Observe that $\gamma_k$ is decreasing, there must be a $K_1$ larger enough such that $ 1 - \lambda\rho_{min}(BB^T) < 1 - D_1\gamma_k$.
Let $K = \max \{ K_0,K_1\}$, then for all $k \geq K$,
\begin{equation}\label{thm2eq2}
\begin{aligned}
& \mathbb{E}_{k + 1}[\lVert x_{k + 1} - x^* \rVert_2^2 + \tfrac{\gamma_{k + 1}^2}{\lambda}\lVert v_{k + 1}  - v^* \rVert_2^2] \\
& \leq (1 - D_1\gamma_k)\mathbb{E}_{k}[\lVert x_k - x^* \rVert_2^2] + \tfrac{\gamma_{k}^2}{\lambda}(1 - \lambda \rho_{min}(BB^T)) \mathbb{E}_{k}[\lVert v_k - v^* \rVert_2^2] + \gamma_k^2D_2 \\
& \leq (1 - D_1\gamma_k)\mathbb{E}_{k}[\lVert x_k - x^* \rVert_2^2 + \tfrac{\gamma_{k}^2}{\lambda}\lVert v_k - v^* \rVert_2^2] + \gamma_k^2D_2 \\
& = (1 - \tfrac{2}{k + 1/D_1})\mathbb{E}_{k}[\lVert x_k - x^* \rVert_2^2 + \tfrac{\gamma_{k}^2}{\lambda}\lVert v_k - v^* \rVert_2^2] + \tfrac{4D_2/D_1^2}{(k + 1/D_1)^2}  . 
\end{aligned}  
\end{equation}
The result follows by setting  $a = 4D_2/D_1^2,c = 1/D_1$ in Lemma {\ref{rate}}.
This completes the proof.
\end{proof}
The Theorem \ref{thm2} states that the FPDFP can get $\mathcal{O}(1/k)$ convergence rate which coincides with FedPAQ. However, FPDFP deals with more complex regularizer without the need to solve subproblems.

\section{Numerical Experiments}\label{sec:experiment}

In this section, we present numerical experiments to verify the performance of the proposed algorithm. 

\subsection{Graph-Guided Logistic Regression}

In this part, we consider the graph-guided \cite{PDFP} logistic regression, which the optimization is given as follows:
\begin{equation}\label{GGLST}
\underset{x \in \mathbb{R}^d}{\min}~\frac{1}{n} \sum_{i = 1}^{n} \log(1 + \exp(-b_is_i^Tx) + \frac{\mu_1}{2}\lVert x \rVert_2^2) + \mu_2\lVert Bx \rVert_1, 
\end{equation}
where the $b_i \in \{-1, 1 \}$ is the label of sample $s_i \in \mathbb{R}^d$. In FL setting, people usually consider the case when $B = I$($I$ is the identity matrix). The Graph-Guided model has shown better generalization ability \cite{SADMM}. We use sparse inverse covariance selection \cite{GLasso,Xray} to obtain graph matrix $G$ and $B = [G;I]$. Two real-world data sets \textit{a9a} and \textit{covtype} from LIBSVM \cite{LIBSVM} are considered, with details of the data sets and the regularization parameters $\mu_1$ and $\mu_2$ provided in Table {\ref{datasets}}. 
For each case, the algorithmic parameter $\lambda$ is set as $\frac{1}{\rho_{\max}{(BB^T)}}$, and the step size $\gamma_k$ is tuned such that the best performance is obtained. The experiments are carried out on a Laptop with i7 11850 processor, Nvidia graphics cards RTX 3070 (8G) with 5888 CUDA cores and 16GB RAM. The version of MATLAB is 2021a.

\begin{table}[htp]
\setlength{\belowcaptionskip}{0.2cm} 
\caption{Summary of different data sets.}
\centering
\begin{tabular}{|c|c|c|c|c|c|c|c|c|c|c|c|c|c|}
\hline
\multicolumn{2}{|c|}{Data sets}
& \multicolumn{2}{|c|}{$\sharp$ of samples}
& \multicolumn{2}{|c|}{$\sharp$ of train}
& \multicolumn{2}{|c|}{$\sharp$ of test}
& \multicolumn{2}{|c|}{$\sharp$ of features}
& \multicolumn{2}{|c|}{$\mu_1$}
& \multicolumn{2}{|c|}{$\mu_2$}\\
\hline
\multicolumn{2}{|c|}{a9a}
& \multicolumn{2}{|c|}{$32,561$}
& \multicolumn{2}{|c|}{$26,053$}
& \multicolumn{2}{|c|}{$6,508$}
& \multicolumn{2}{|c|}{$123$}
& \multicolumn{2}{|c|}{$10^{-5}$}
& \multicolumn{2}{|c|}{$10^{-5}$}\\
\hline
\multicolumn{2}{|c|}{covtype}
& \multicolumn{2}{|c|}{$581,012$}
& \multicolumn{2}{|c|}{$464,809$}
& \multicolumn{2}{|c|}{$116,203$}
& \multicolumn{2}{|c|}{$54$}
& \multicolumn{2}{|c|}{$10^{-5}$}
& \multicolumn{2}{|c|}{$10^{-5}$}\\
\hline
\end{tabular}
\label{datasets}
\end{table}

In Figures {\ref{GGLST1}} and {\ref{GGLST2}}, we provide the relative error of training loss(y-axis in log scale), testing loss(y-axis in log scale), and testing accuracy(log-log scale) of two data sets, under different quantization levels and participation number $n$. 
It can be observed from the figures that with fixed $n$, a larger quantization level leads to more accurate solutions (Figure {\ref{GGLST1}} and {\ref{GGLST2}} (a-c)). The reason is that a larger quantization level can have less variance and allow us to use a larger decreasing step size and thus obtain a better solution. 
\begin{itemize}
\item It can also be observed that a proper quantization level (e.g., $s = 20$ in Figure {\ref{GGLST1}} and {\ref{GGLST2}} (a-c)) can achieve a performance identical to that without quantization, which indicates that  quantization indeed can  reduce the communication overhead without hurting the quality of the solution. 

\item With quantization level $s = 20$ fixed, it can be observed from the figures that a small participation number can achieve faster convergence but gives a lower accurate solution (Figure {\ref{GGLST1}}, (d-f), $n = 10$; Figure {\ref{GGLST2}}, (d-f), $n = 5$). 

\item It is also unnecessary to require all clients to participate in the aggregation. Setting $n = 20$ (for both cases) is sufficient to get a solution similar to that when all clients have participated. 
\end{itemize}
Overall, the combination of quantization and partial participation can make the iterates converge to a relatively similar solution more efficiently. In addition, the plots of the number of bits versus the communication rounds for different scenarios are provided in Figure {\ref{GGLST3}}. As we can see, appropriate quantization level and participation number lead to lower communication costs with few bits required.


\begin{figure}[htbp!]
\centering
\subfigure[a9a (train)]{
\centering
\includegraphics[width = 1.8 in ]{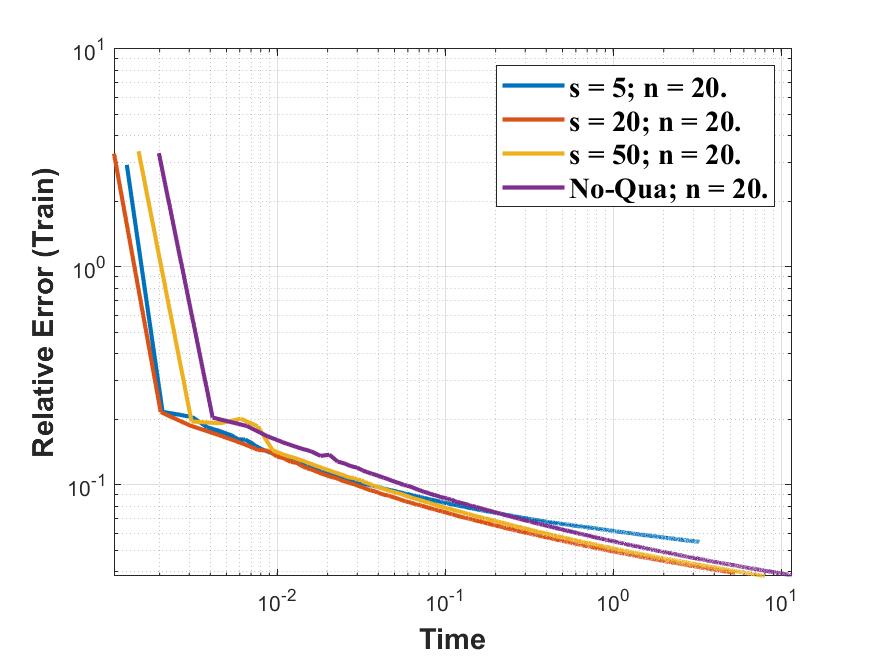}
}
\subfigure[a9a (test)]{
\centering
\includegraphics[width = 1.8 in]{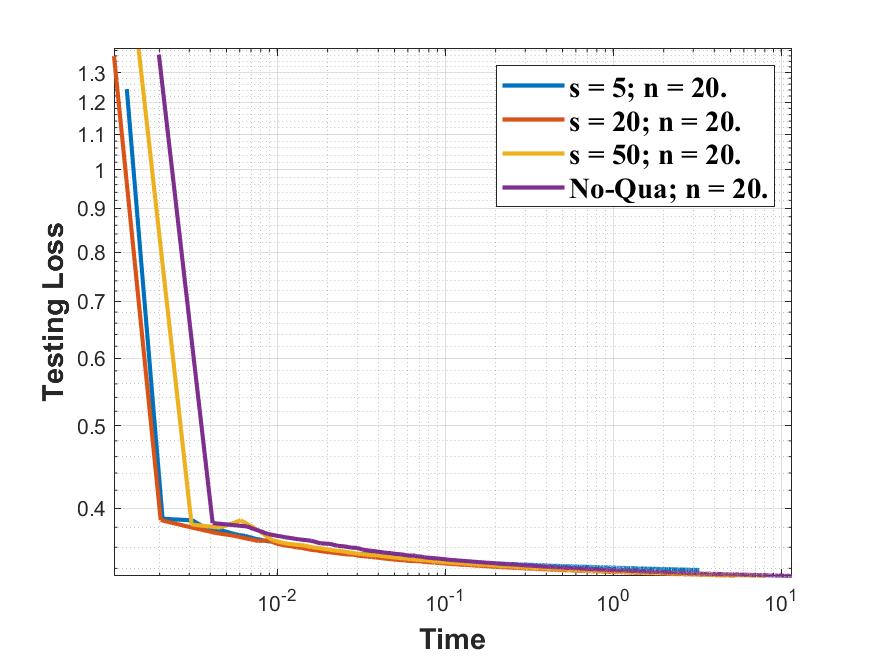}
}
\subfigure[a9a (accuracy) ]{
\centering
\includegraphics[width = 1.8 in]{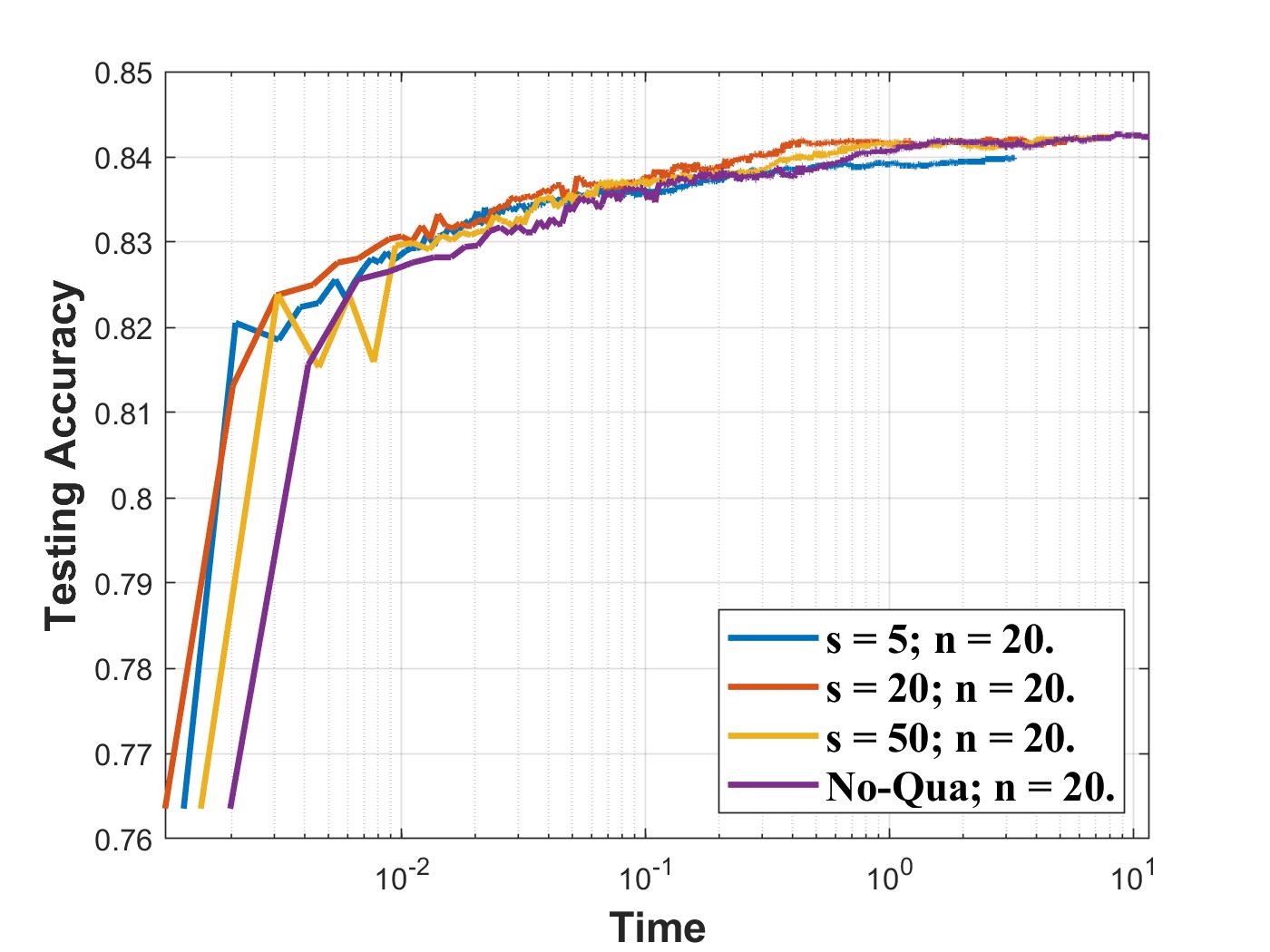}
}\hspace{1mm}

\centering
\subfigure[a9a (train)]{
\centering
\includegraphics[width = 1.8 in ]{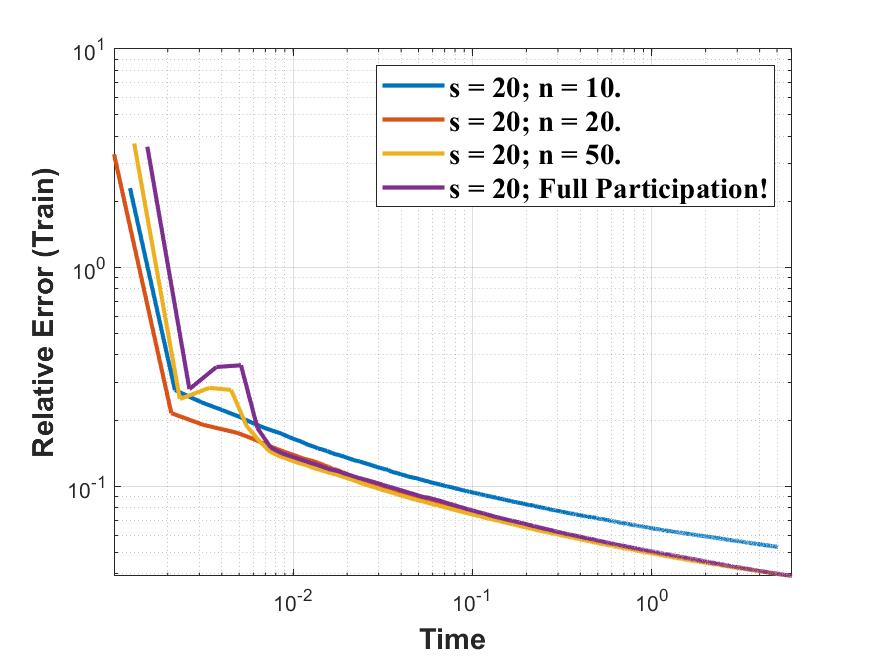}
}
\subfigure[a9a (test)]{
\centering
\includegraphics[width = 1.8 in]{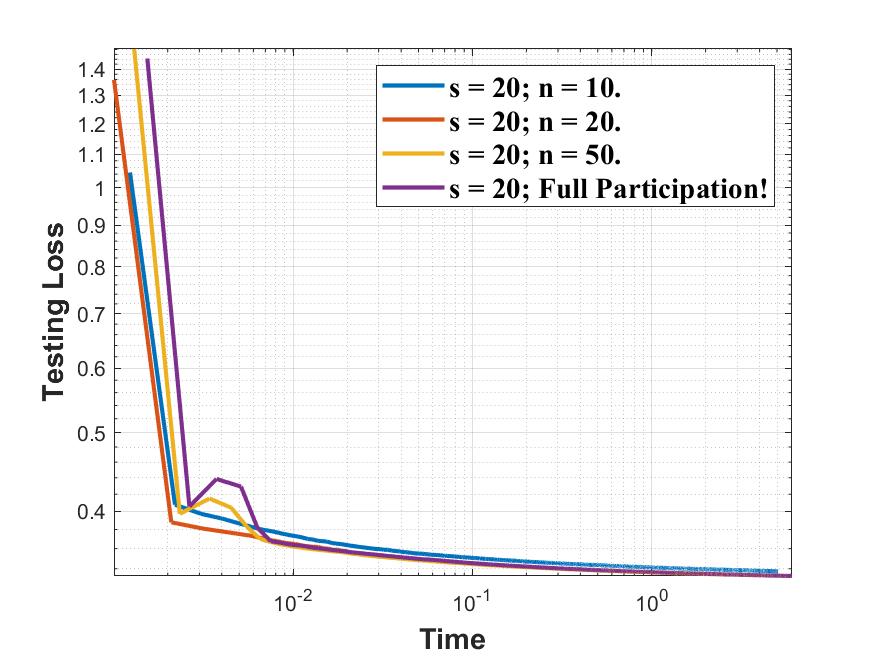}
}
\subfigure[a9a (accuracy)]{
\centering
\includegraphics[width = 1.8 in]{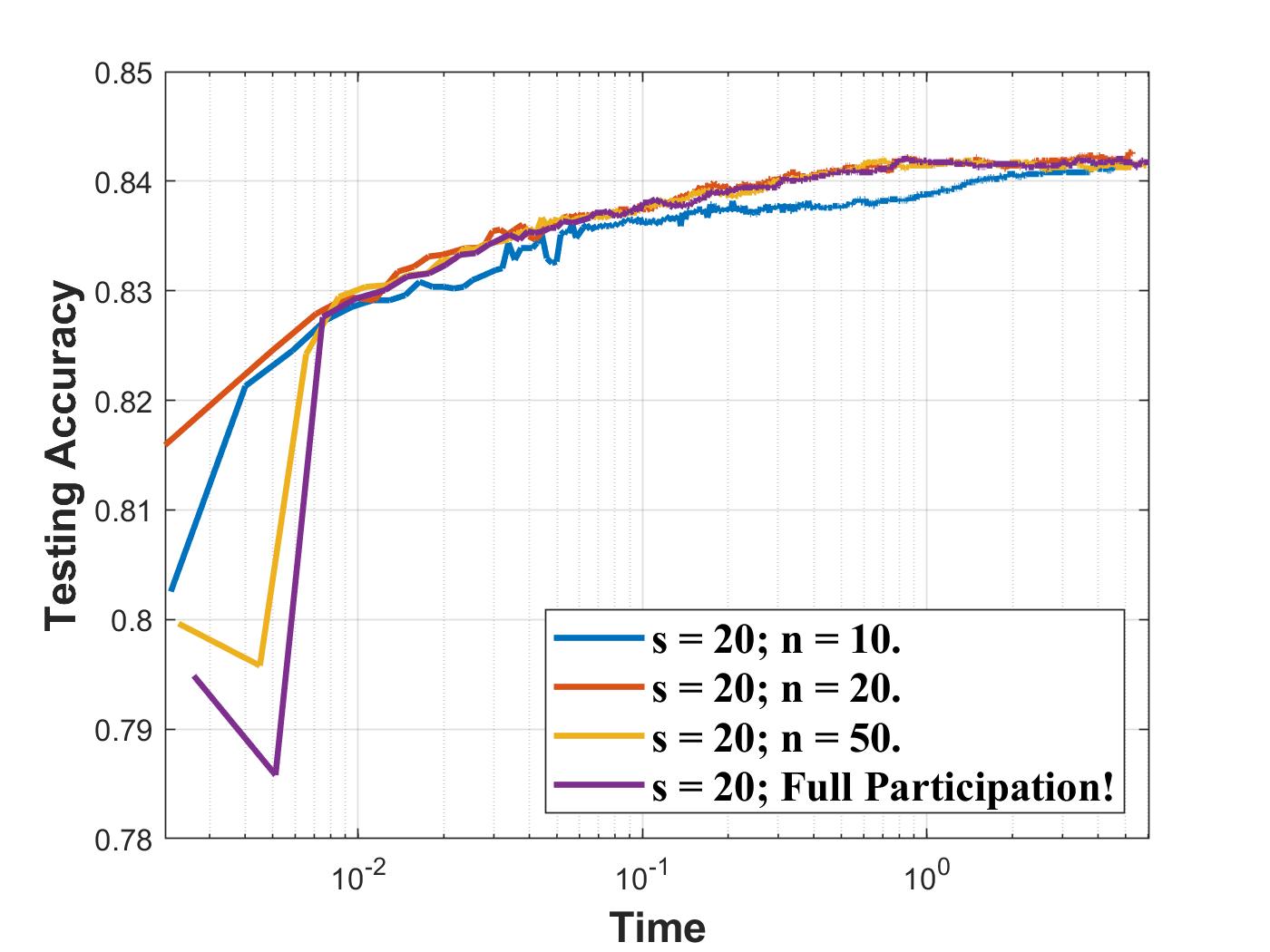}
}\hspace{1mm}
\caption{Training, testing loss and testing accuracy for data set \textit{a9a} for different quantization level $s$ and participation number $n$.}
\label{GGLST1}
\end{figure} 

\begin{figure}[htbp!]
\centering
\subfigure[covtype (train)]{
\centering
\includegraphics[width = 1.8 in ]{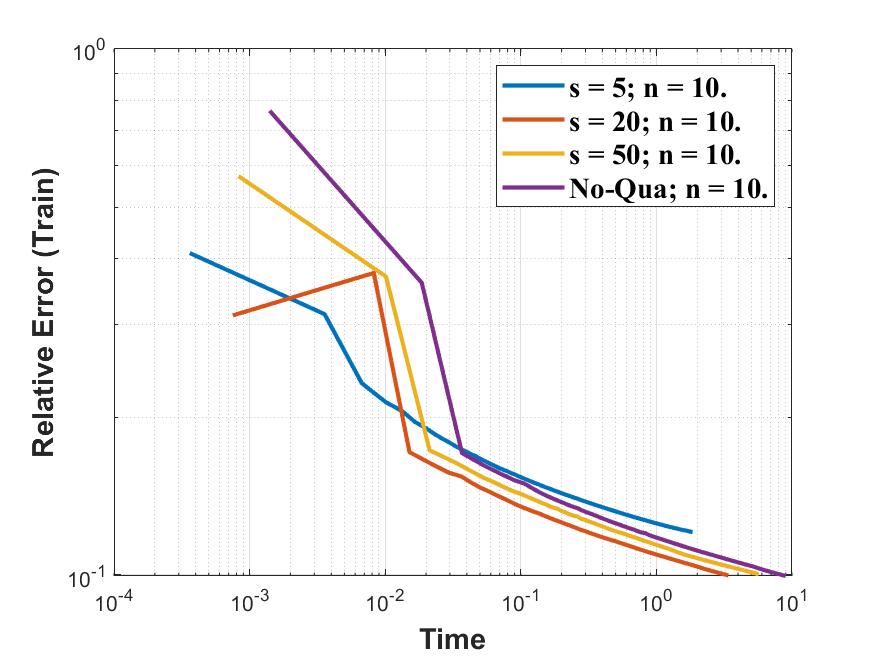}
}
\subfigure[covtype (test)]{
\centering
\includegraphics[width = 1.8 in]{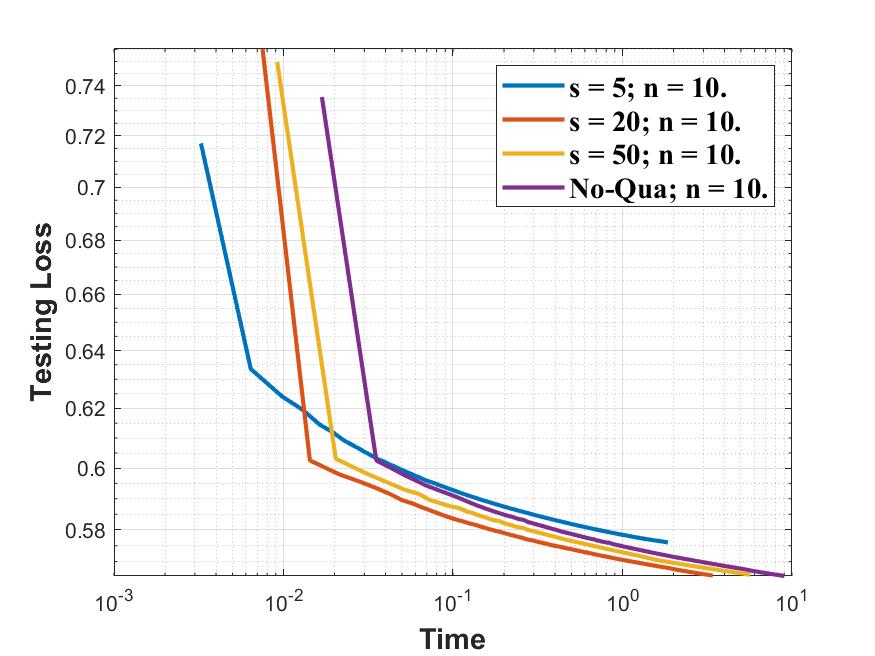}
}
\subfigure[covtype (accuracy)]{
\centering
\includegraphics[width = 1.8 in]{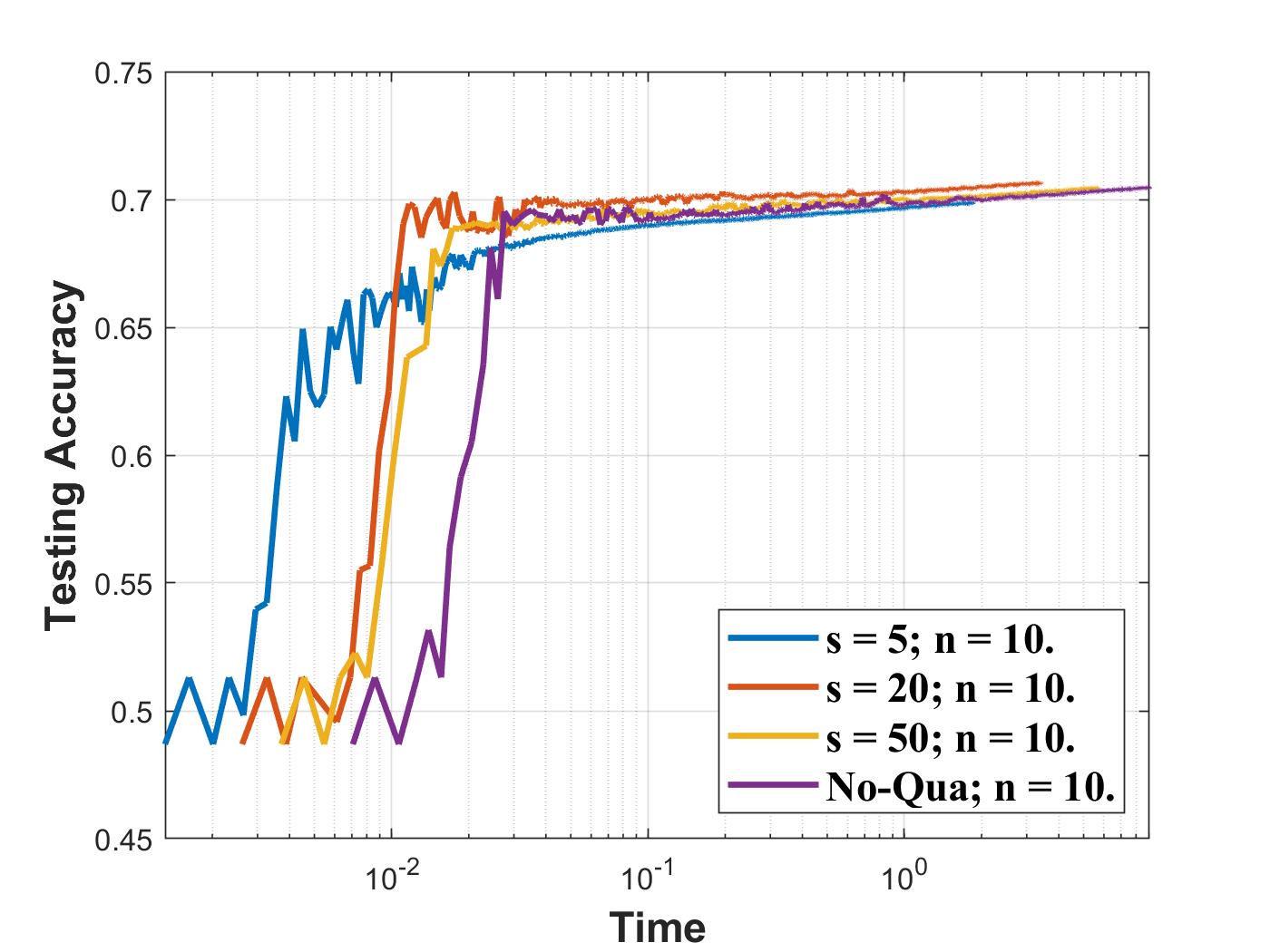}
}\hspace{1mm}

\centering
\subfigure[covtype (train)]{
\centering
\includegraphics[width = 1.8 in ]{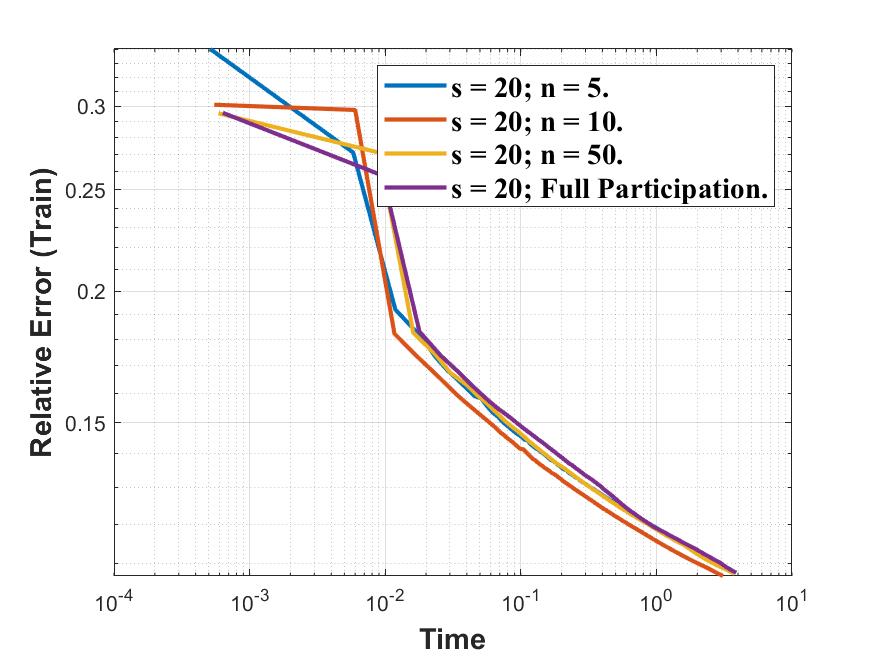}
}
\subfigure[covtype (test)]{
\centering
\includegraphics[width = 1.8 in]{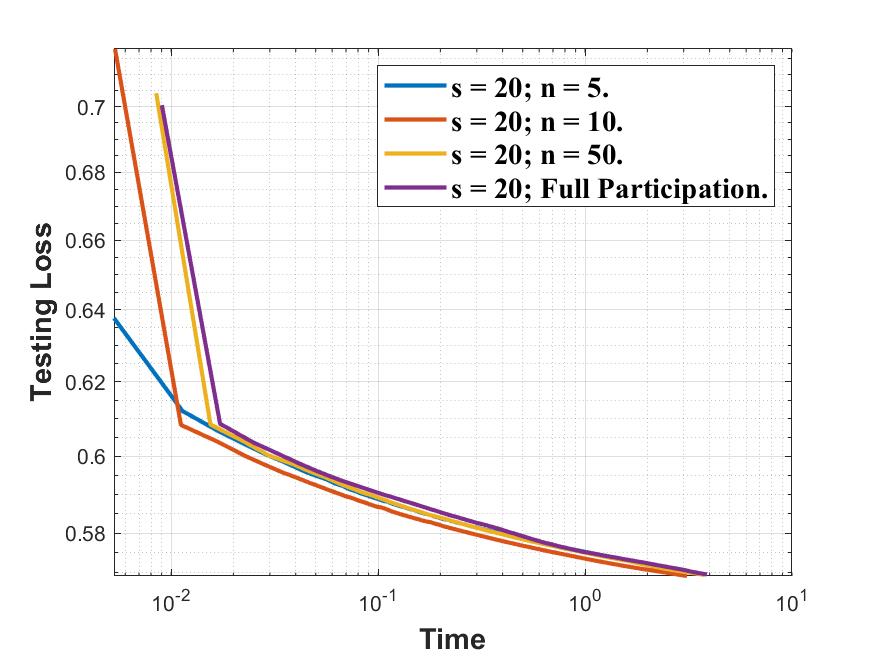}
}
\subfigure[covtype (accuracy) ]{
\centering
\includegraphics[width = 1.8 in]{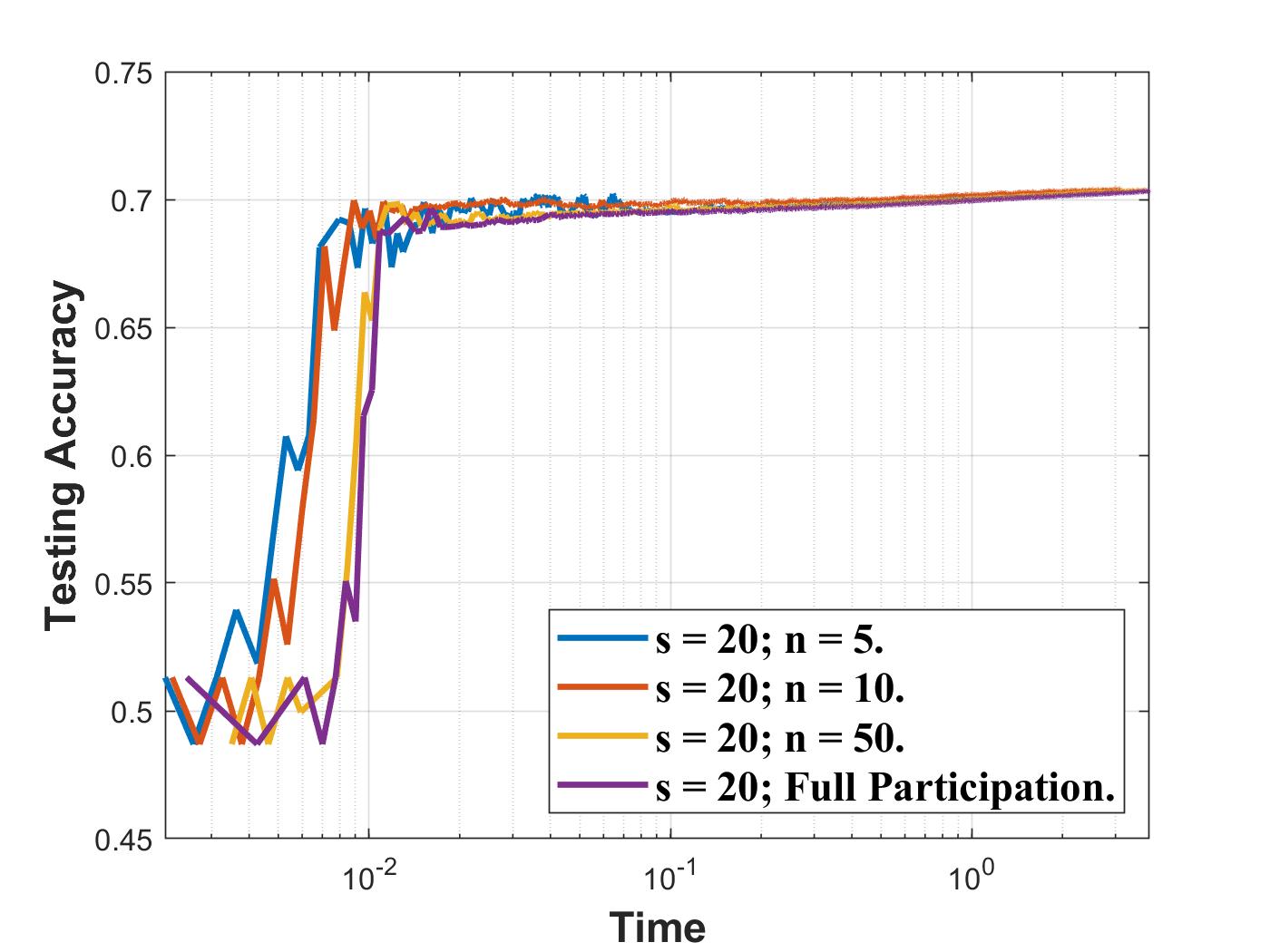}
}\hspace{1mm}
\caption{Training, testing loss and testing accuracy for data set \textit{covtype} under different quantization level $s$ and participation number $n$.}
\label{GGLST2}
\end{figure} 

\begin{figure}[htbp!]
\centering
\subfigure[communication bits (a9a)]{
\centering
\includegraphics[width = 2.5 in ]{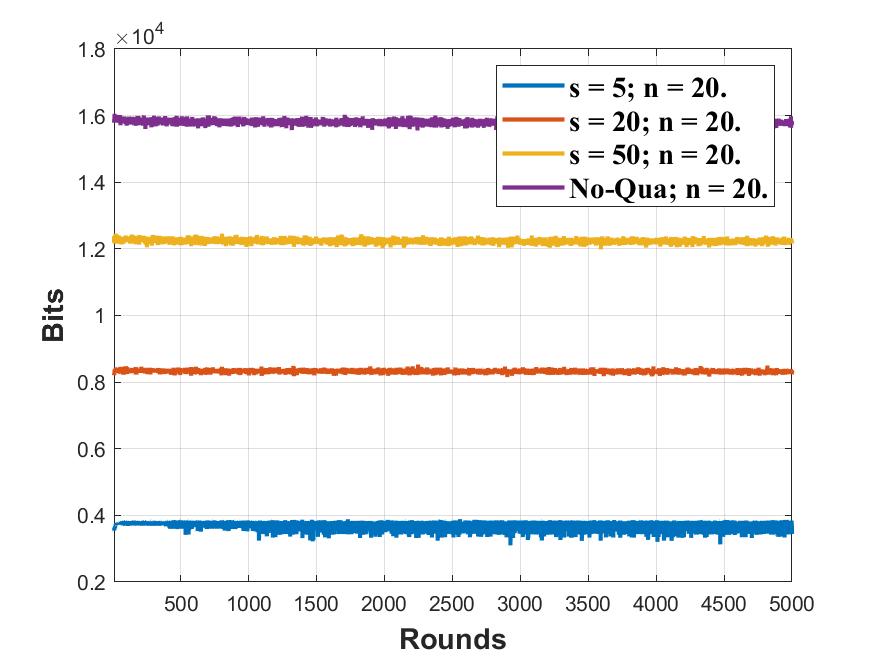}
}
\subfigure[communication bits (a9a)]{
\centering
\includegraphics[width = 2.5 in]{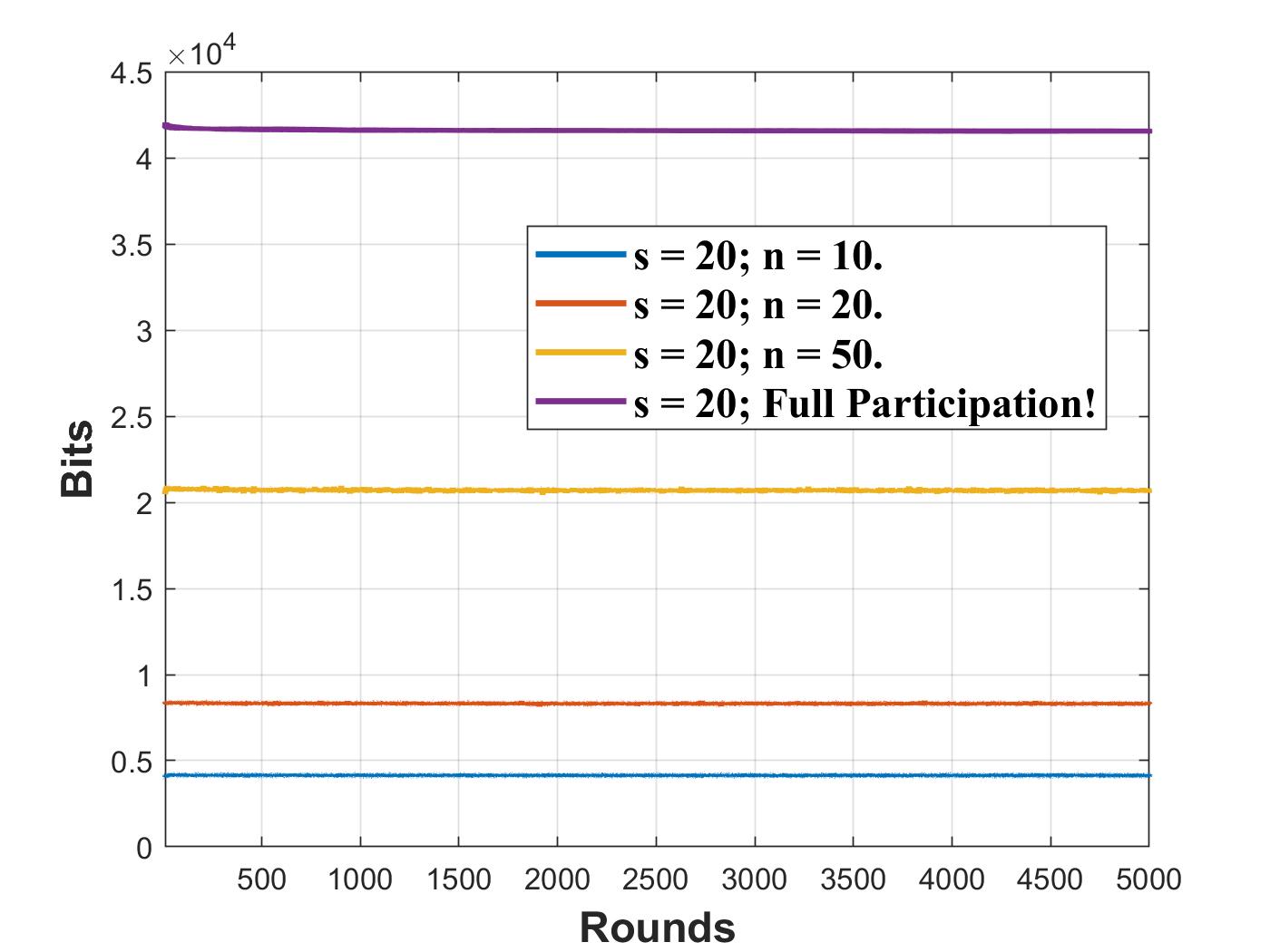}
}

\subfigure[communication bits (covtype)]{
\centering
\includegraphics[width = 2.5 in ]{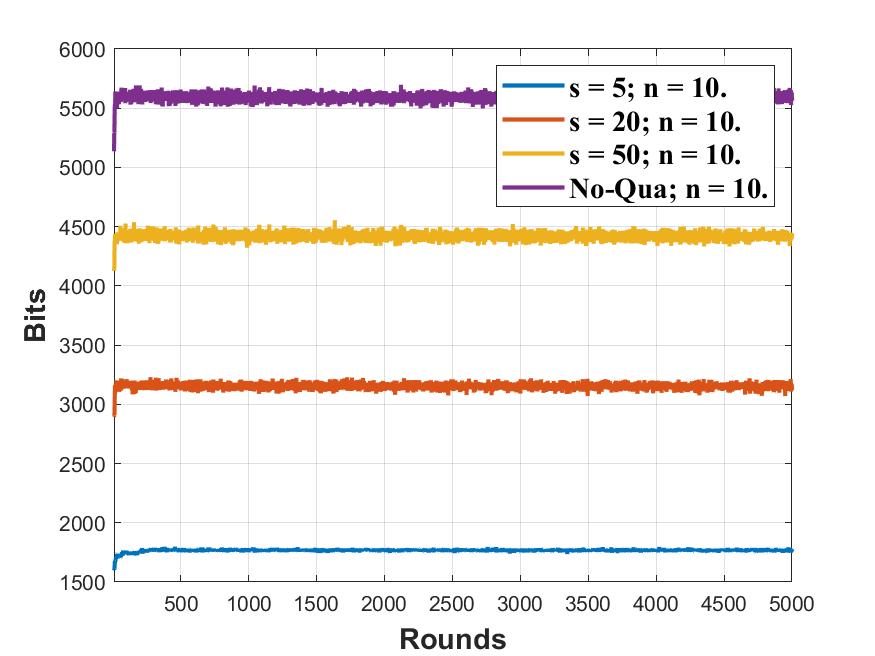}
}
\subfigure[communication bits (covtype)]{
\centering
\includegraphics[width = 2.5 in]{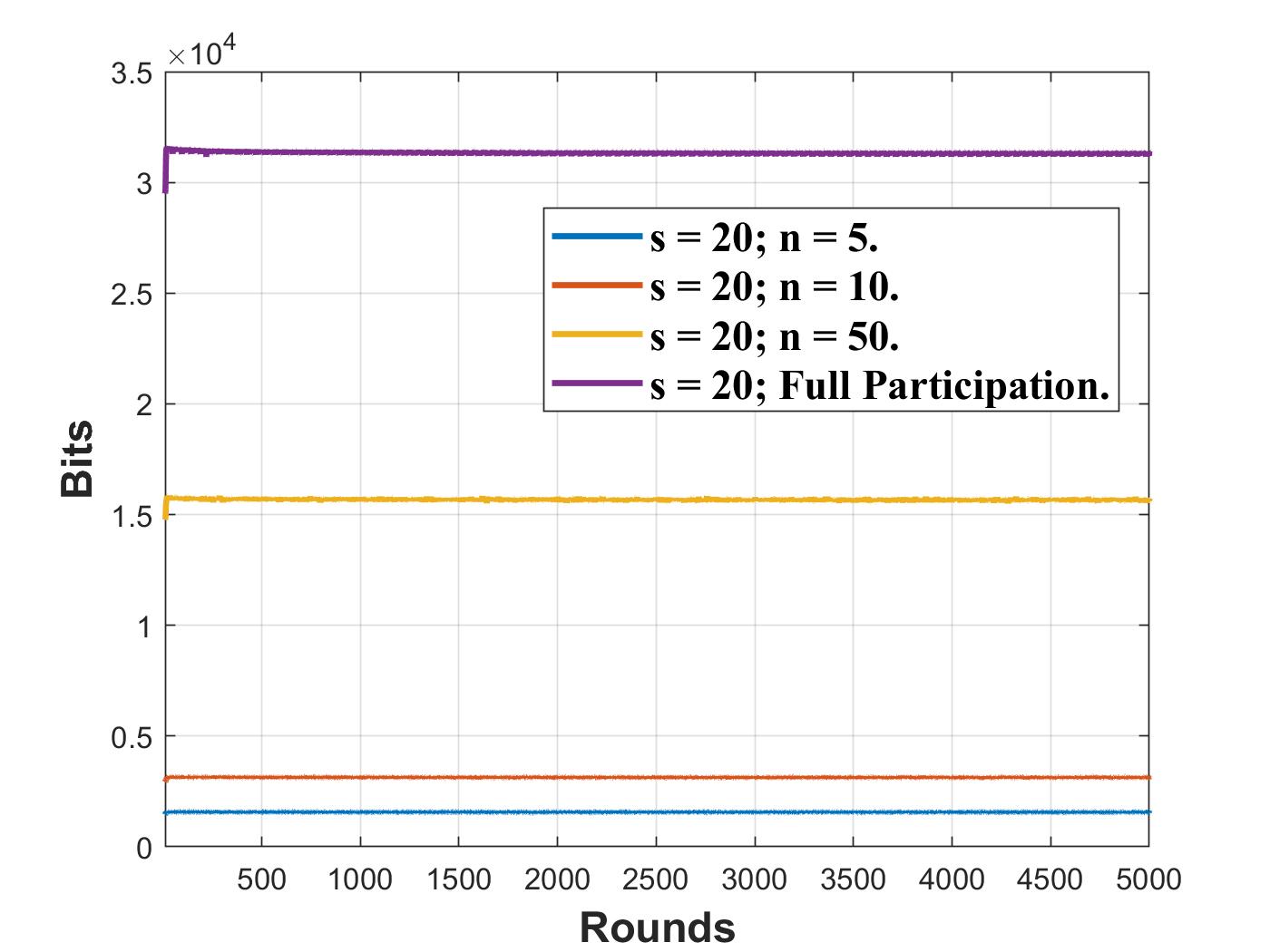}
}
\caption{The communication bits versus the communication round.}
\label{GGLST3}
\end{figure}

\subsection{3D Computed Tomography Reconstruction}
For medical image reconstruction tasks, the acquisition of high-quality images usually requires a huge amount of projection data. The limited computational capability of a single machine brings a challenge to reconstruction efficiency. One way to address the issue is to use distributed computation framework \cite{multirecon1,multirecon2,multirecon3}. In this section, we extend 3D CT reconstruction to federated setting to evaluate the proposed FPDFP. We consider the following TV-$\ell_2$ model 

\begin{equation}\label{TVl2}
 \min_{x \in \mathbb{R}^d}~\frac{1}{2}\lVert \mathcal{A}x - b \rVert_2^2 + \mu \lVert \nabla x \rVert_{1,2}, 
\end{equation}
where 
\begin{itemize}
  \item $x \in \mathbb{R}^d$ is the vectorized image to be reconstructed with dimension  $d = 256\times256\times64$.
  \item $\mathcal{A}$ is  discrete Radon transform. The size of projection plane is $n_a \times n_b = 512 \times 384$, and the number of viewer is $n_v = 668$. Since the scale of this example is relatively large we use the parallelization method proposed in \cite{Xray} to compute Radon transform and its adjoint. Here we distribute the angles of projections among $N=17$ clients in a i.i.d fashion.
  \item $b$ is the noisy projection vector which is obtained by adding a Gaussian noise $\varepsilon$ with zero mean and variance $0.02$ to the projection data of ground truth image $x_0$, i.e., $b = \mathcal{A}x_0 + \varepsilon$.
  \item $\mu = 10^{-3}$ is the regularizer parameter.
  \item $\nabla$ is the discrete gradient operator, and $\lVert \cdot \rVert_{1,2}$ is the $\ell_{1,2}$ norm, i.e. we use the `isotropic' total variation \cite{TV} for the regularization term.
 
\end{itemize}
The parameter $\lambda$ is set as $\frac{1}{20}$ and the step size $\gamma_k$ are tuned such that the best performance is achieved.  Different quantization level $s = 500,1000,3000,5000$, and participation number $n = 2,5,10$
and full participation $n = 17$ are considered. We note that the quantization level $s$ is much larger than the previous example, which is due to the fact that the dimension of decision variable $x$ is much larger. We need a finer chopped interval so that the variance is within a proper threshold and the algorithm can converge. Inspired by the block quantization proposed in \cite{QSGD}, we quantize each slice of the image with the same quantization level.

The objective function value and Peak Signal to Noise Ratio (PSNR) value versus the number of iteration are depicted in Figure {\ref{CT3D1}} and {\ref{CT3D2}}. With participation number $n = 5$ fixed, it can be seen from Figure {\ref{CT3D1}} that the advantage of quantization is obvious. A smaller quantization level leads to faster convergence but gives a higher objective value (Figure {\ref{CT3D1}} (a) for $s = 500$). $s = 1000$ is sufficient to get a proper solution (see Figure {\ref{CT3D1}} and one slice reconstruction results in {\ref{CT3D3}}). 
With the quantization level $s = 1000$ fixed, it can be observed from Figure {\ref{CT3D2}} and {\ref{CT3D4}} that setting partial participation $n = 5,10$ is enough to reconstruct images with enough quality. The advantage of quantization can be further demonstrated by the plots of bits transferred during communication for the scenarios above.

\begin{figure}[htbp!]
\subfigure[objective value]{
\centering
\includegraphics[width = 2.5 in]{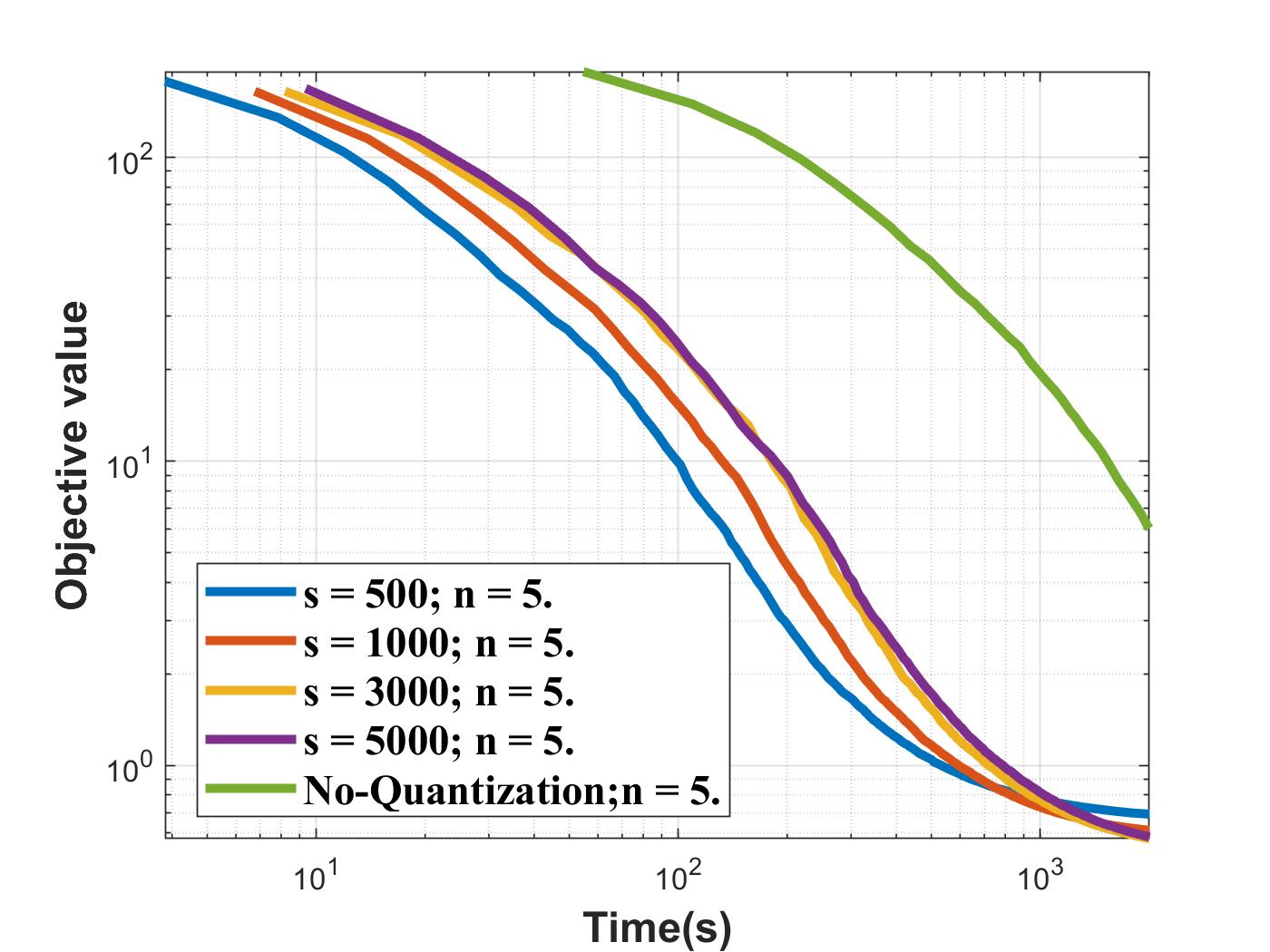}
}
\subfigure[PSNR]{
\centering
\includegraphics[width = 2.5 in]{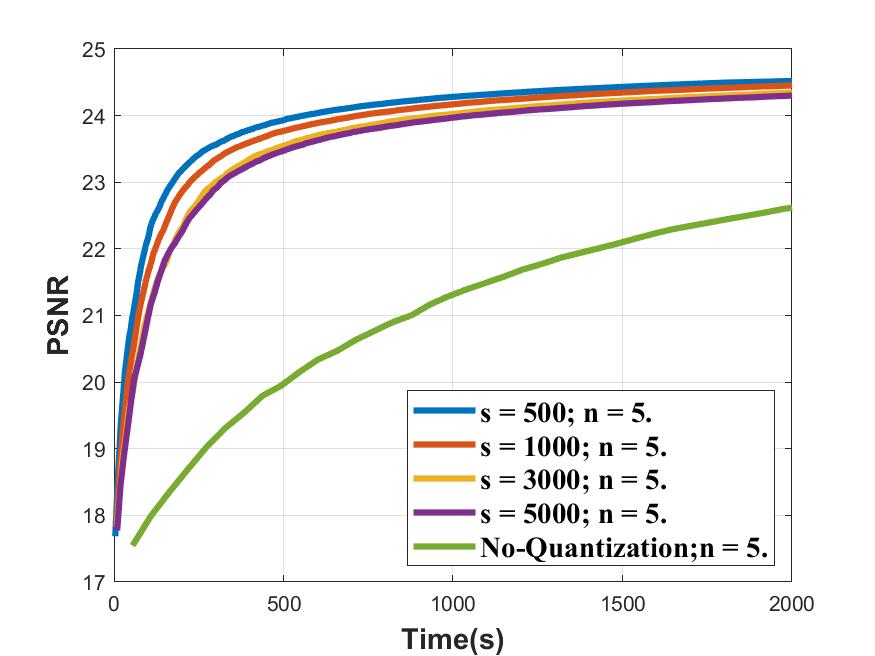}
}
\caption{Objective value and PSNR for different quantization levels.}
\label{CT3D1}
\end{figure} 


\begin{figure}[htbp!]
\subfigure[objective value]{
\centering
\includegraphics[width = 2.5 in]{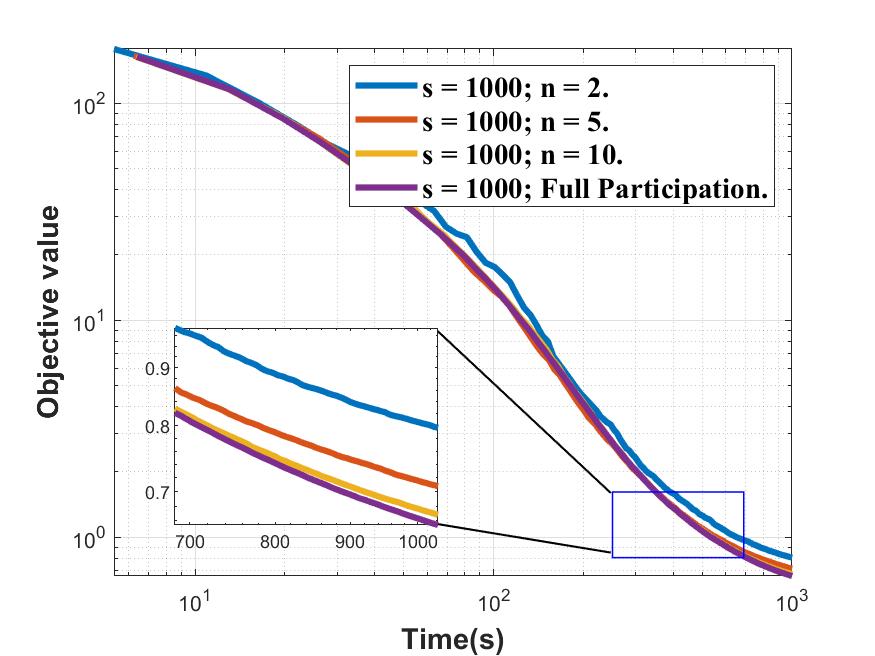}
}
\subfigure[PSNR]{
\centering
\includegraphics[width = 2.5 in]{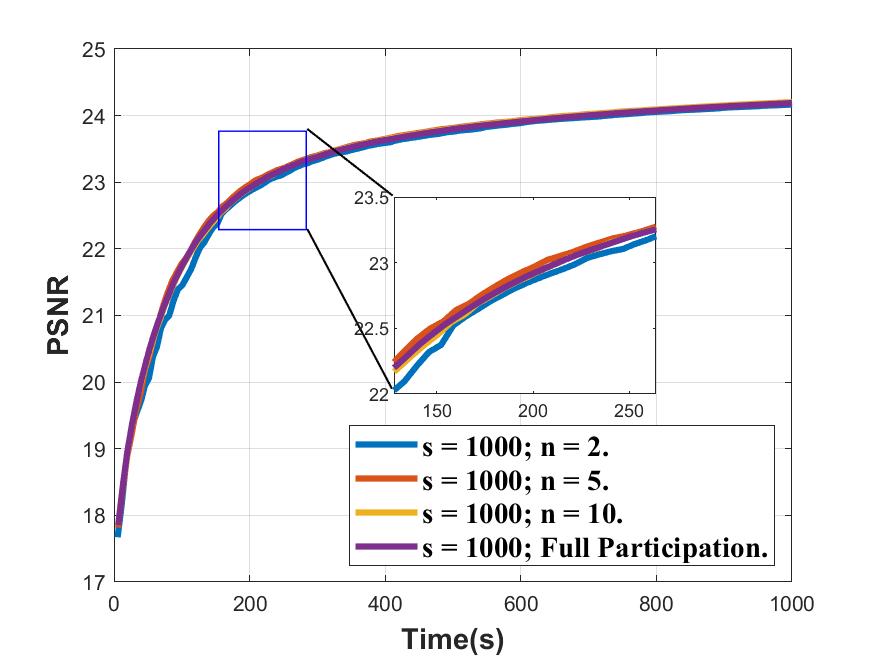}
}\hspace{1mm}
\caption{Objective value and PSNR for different quantization participation number.}
\label{CT3D2}
\end{figure}

\begin{figure}[htbp!]
\subfigure[objective value]{
\centering
\includegraphics[width = 2.5 in]{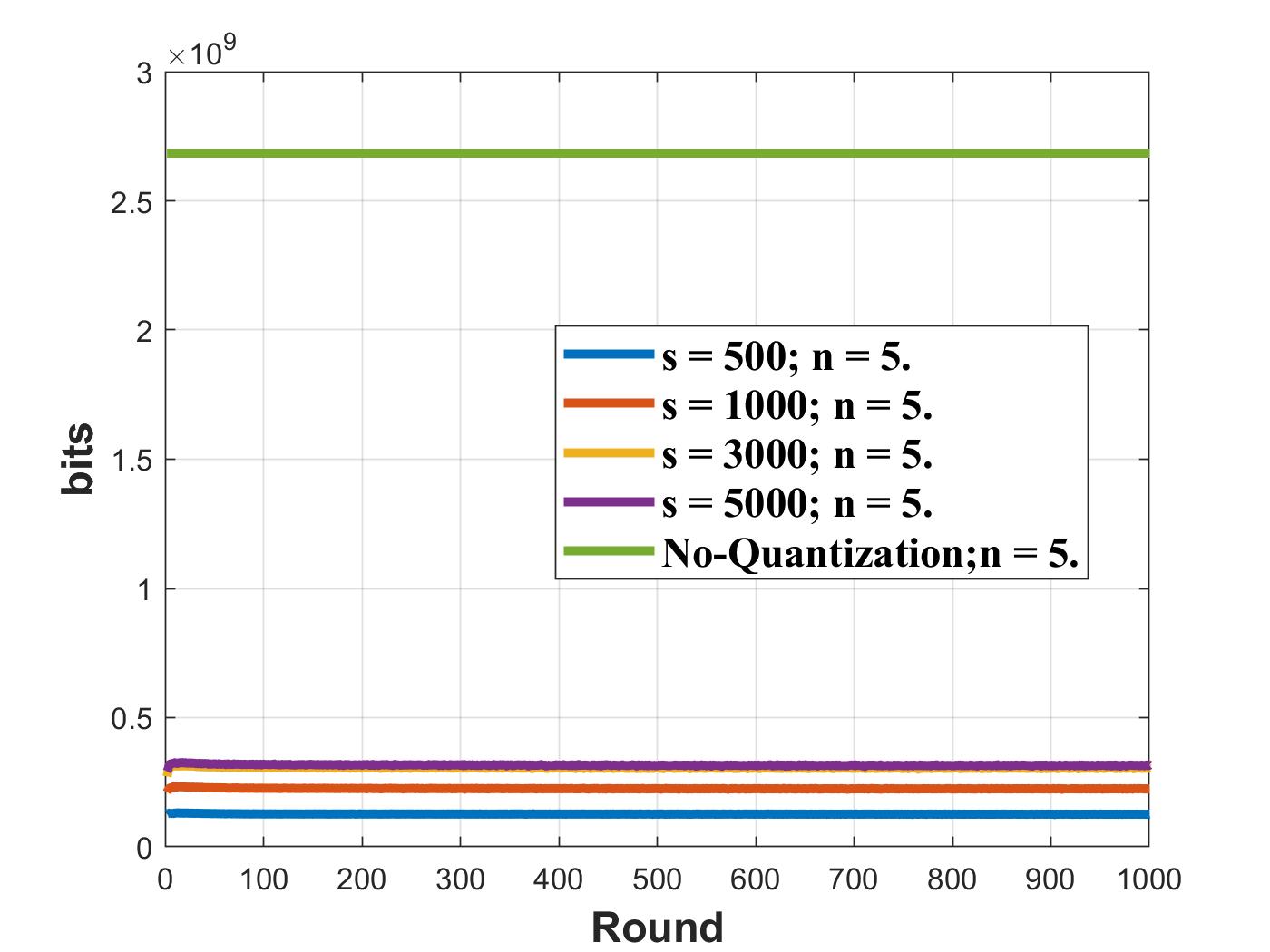}
}
\subfigure[PSNR]{
\centering
\includegraphics[width = 2.5 in]{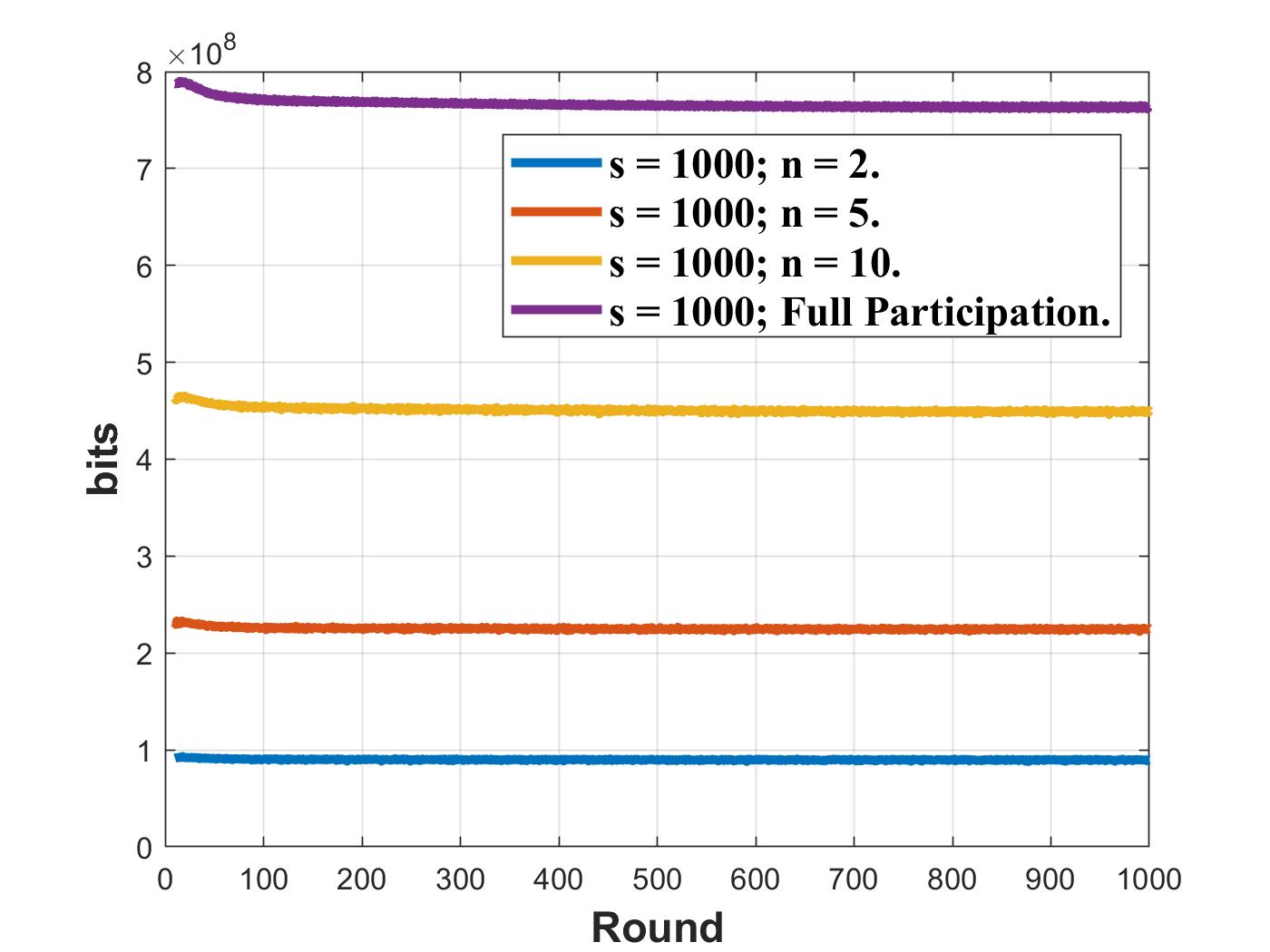}
}\hspace{1mm}
\caption{Objective value and PSNR for different quantization participation number. {\bf (left)} Different quantization level. {\bf (right)} Different participation number.}
\label{CT3DBIT}
\end{figure} 


\begin{figure}[htbp!]
\centering
\subfigure[Ground Truth]{
\centering
\includegraphics[width = 1.8 in ]{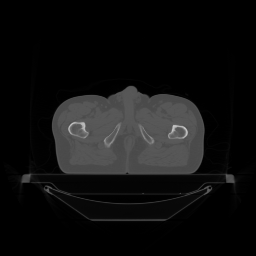}
}
\subfigure[$s = 500, n = 5$, $\mathrm{PSNR} = 44.76$]{
\centering
\includegraphics[width = 1.8 in]{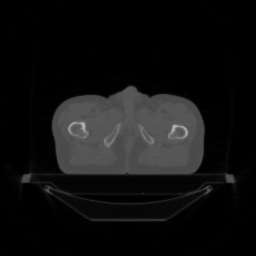}
}
\subfigure[$s = 1000, n = 5$, $\mathrm{PSNR} = 45.78$]{
\centering
\includegraphics[width = 1.8 in]{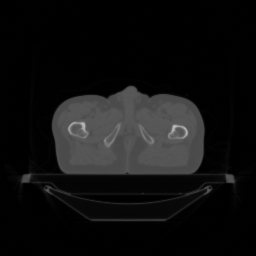}
}
\subfigure[$s = 3000, n = 5$, $\mathrm{PSNR} = 45.83$]{
\centering
\includegraphics[width = 1.8 in]{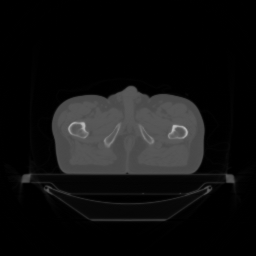}
}
\subfigure[$s = 5000, n = 5$, $\mathrm{PSNR} = 45.86$]{
\centering
\includegraphics[width = 1.8 in]{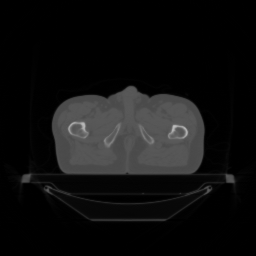}
}
\subfigure[No quantization, $n = 5$, $\mathrm{PSNR} = 45.86$]{
\centering
\includegraphics[width = 1.8 in]{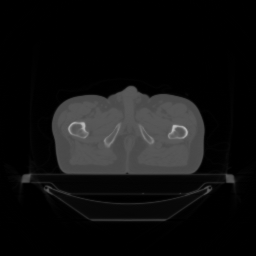}
}\hspace{1mm}
\caption{One slice of the reconstructed image.}
\label{CT3D3}
\end{figure}

\begin{figure}[htbp!]
\centering
\subfigure[Ground Truth]{
\centering
\includegraphics[width = 1.8 in ]{GT.png}
}
\subfigure[$s = 1000, n = 2$, $\mathrm{PSNR} = 45.10$]{
\centering
\includegraphics[width = 1.8 in]{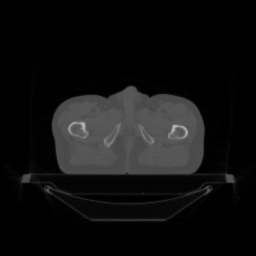}
}
\subfigure[$s = 1000, n = 5$, $\mathrm{PSNR} = 45.78$]{
\centering
\includegraphics[width = 1.8 in]{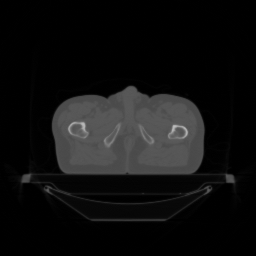}
}
\subfigure[$s = 1000, n = 10$, $\mathrm{PSNR} = 45.92$]{
\centering
\includegraphics[width = 1.8 in]{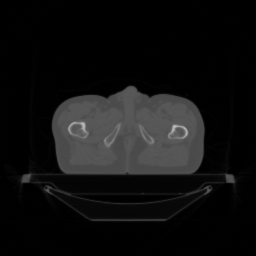}
}
\subfigure[$s = 1000$, full participation, $\mathrm{PSNR} = 45.82$]{
\centering
\includegraphics[width = 1.8 in]{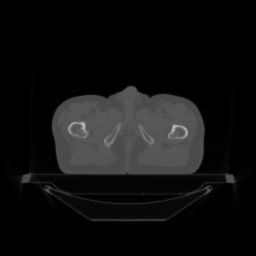}
}\hspace{1mm}
\caption{One slice of the reconstructed image.}
\label{CT3D4}
\end{figure}

\section{Conclusion}
In this paper, we proposed a federated primal dual fixed point (FPDFP) method to solve separable linearly composite convex optimization problems. We additionally combined quantization and partial participation to reduce communication overhead during the learning process. Theoretically, we established the convergence and convergence rate of FPDFP on some standard assumptions and validated the proposed algorithm by graph-guided logistic regression and 3D CT reconstruction. 


\section{Appendix}\label{sec:appendix}
\subsection{Appendix A}
Before proceeding with the proof of the main theorem, we first introduce some preliminary lemmas. Lemma {\ref{lm1}, \ref{Moreau}, \ref{lm2}} are similar to Lemma {7.1, 3.2, 4.1} of \cite{SPDFP}. We include the proof here for our presentation to be self-contained. 

\begin{lemma}\label{Moreau}
Let $r > 0$, $f_0(x), x \in \mathbb{R}^d$ be proper convex l.s.c. and $h(x) = r f_0(x/r)$, then for any$ ~ y \in \mathbb{R}^d$, it holds that $\mathrm{Prox}_h(y) = r \mathrm{Prox}_{r^{-1}f_0}(y/r)$.
\end{lemma}
\begin{proof}
The assertion can be proved by using the definition of $\mathrm{Prox}_{f_0}(\cdot)$ and change of variables.
\end{proof}


\begin{proof}[Proof of Lemma {\ref{lm1}}]
By the first optimality condition of problem {(\ref{pb1})}, we have
\begin{equation}\label{apeq}
\begin{aligned}
x^{*}
= \underset{x \in \mathbb{R}^d}{\arg\min} f(x) + g \circ B (x)  
&\quad\Leftrightarrow\quad 0 \in -\nabla f(x^*) - \partial (g \circ B )(x^*) \\
&\quad\Leftrightarrow\quad 0 \in -\gamma_k \nabla f(x^*) - \gamma_k \partial (g \circ B )(x^*) \\
&\quad\Leftrightarrow\quad x^*  \in x^* -\gamma_k \nabla f(x^*) - \gamma_k B^T \partial g(Bx^*) \\
&\quad\Leftrightarrow\quad x^*  \in x^* -\gamma_k \nabla f(x^*) - \lambda \big(B^T \circ \tfrac{\gamma_k}{\lambda} \partial g(Bx^*).
\end{aligned}
\end{equation}
\noindent
Let $v^* \in \partial g(Bx^*)$, 
then {(\ref{apeq})} can then be rewritten as
\begin{equation}\label{apeq5}
x^* = x^* - \gamma_k \nabla f(x^*) - \gamma_k B^T v^* .
\end{equation}
Furthermore, we have $\frac{\gamma_k}{\lambda}v^* \in \partial \frac{\gamma_k}{\lambda}g(Bx^*)$,
which is equivalent to
\begin{equation}\label{apeq2}
\begin{aligned}
Bx^* = \mathrm{Prox}_{\frac{\gamma_k}{\lambda}g}(Bx^* + \tfrac{\gamma_k}{\lambda}v^*) 
&\quad\Leftrightarrow\quad (Bx^* + \tfrac{\gamma_k}{\lambda}p^*) - \tfrac{\gamma_k}{\lambda}v^*  = \mathrm{Prox}_{\frac{\gamma_k}{\lambda}g}(Bx^* + \tfrac{\gamma_k}{\lambda}v^*) \\
&\quad\Leftrightarrow\quad \tfrac{\gamma_k}{\lambda}v^* = (I - \mathrm{Prox}_{\frac{\gamma_k}{\lambda}g})(Bx^* + \tfrac{\gamma_k}{\lambda}v^*) \\
&\quad\Leftrightarrow\quad v^* = \tfrac{\lambda}{\gamma_k}(I - \mathrm{Prox}_{\frac{\gamma_k}{\lambda}g})(Bx^* + \tfrac{\gamma_k}{\lambda}v^*) \\
&\quad\Leftrightarrow\quad v^* = \tfrac{\lambda}{\gamma_k}Bx^* + v^* - \tfrac{\lambda}{\gamma_k}\mathrm{Prox}_{\frac{\gamma_k}{\lambda}g}(Bx^* + \tfrac{\gamma_k}{\lambda}v^*) \\
&\quad\Leftrightarrow\quad v^* = \tfrac{\lambda}{\gamma_k}Bx^* + v^* - \tfrac{\lambda}{\gamma_k}\mathrm{Prox}_{(\frac{\lambda}{\gamma_k})^{-1}g}\big(\tfrac{\tfrac{\lambda}{\gamma_k}Bx^* + v^*}{\tfrac{\lambda}{\gamma_k}}\big) \\
&\quad\Leftrightarrow\quad v^* = \mathrm{Prox}_{\frac{\lambda}{\gamma_k} g^*}\big(\tfrac{\lambda}{\gamma_k}Bx^* + v^*\big).
\end{aligned}
\end{equation}
where the last equality follows from the Moreau identity \cite{cvxbook} and $g^*$ is the conjungate function of $g$. 

Inserting {(\ref{apeq5})} into the last equality of {(\ref{apeq2})} yields
\begin{equation}\label{sumup}
\begin{aligned}
v^* & = \mathrm{Prox}_{\frac{\lambda}{\gamma_k} g^*}\big(\tfrac{\lambda}{\gamma_k}B(x^* - \gamma_k \nabla f(x^*)) + (I - \lambda BB^T)v^* \big).  
\end{aligned}
\end{equation}
Combining {(\ref{apeq5})}, {(\ref{apeq2})} and {(\ref{sumup})}, and the definition of $T_k$, we arrive at 
\begin{equation*}
\left\{
\begin{aligned}
v^* & = \mathrm{Prox}_{\frac{\lambda}{\gamma_k} g^*}\big(\tfrac{\lambda}{\gamma_k}B(x^* - \gamma_k \nabla f(x^*)) + (I - \lambda BB^T)v^* \big) \\
& = T_k(x^*,v^*) \\
x^* & = x^* - \gamma_k \nabla f(x^*) - \gamma_k B^T T_k(x^*,v^*).
\end{aligned}
\right.
\end{equation*}
The converse can be similarly verified, and we conclude the proof.
\end{proof}

\begin{lemma}\label{lm2}
Suppose Assumptions {\ref{smooth},\ref{VR1}} and {\ref{VR2}} hold, let $0 < \lambda \leq 1/\rho_{\max}(BB^T)$ and $(x_k^{(i)},v_k^{(i)}),$ $i = 1,\cdots,N$ be the iterates of $i$'th client in  Algorithm \ref{FPDFP} and $(x^*,v^*)$ as in Lemma {\ref{lm1eq1}}, then
\begin{equation}\label{lm2eq0}
\begin{aligned}
& \mathbb{E}_{R_{k}}[\lVert \overline{x}_{k + 1} - x^* \rVert_2^2 + \tfrac{\gamma_{k + 1}^2}{\lambda}\lVert \overline{v}_{k + 1}  - v^* \rVert_2^2 | \mathcal{F}_k] \\
& \leq  (1 + \tfrac{3\gamma_k^2}{\beta^2} )\lVert x_k - x^* \rVert_2^2 + \tfrac{\gamma_{k}^2}{\lambda} \lVert v_k - v^* \rVert_M^2
 - 2\gamma_k \big< \nabla f(x_k) - \nabla f(x^*), x_k - x^*\big> +3 \gamma_k^2(\delta^2 + \sigma^2),\\
\end{aligned}
\end{equation}
where $\mathbb{E}_{R_{k}}[\cdot | \mathcal{F}_k] = \prod_{i = 1}^N \mathbb{E}_{R_{k}}^{(i)}[\cdot | \mathcal{F}_k]$ and $\mathbb{E}_{R_{k}}^{(i)}[\cdot | \mathcal{F}_k]$ denotes the expectation of stochastic gradient update conditioned on the randomness up to the $k$'th round $\mathcal{F}_k$ for client $i$ . $\overline{x}_{k + 1} = \frac{1}{N}\sum_{i = 1}^N x_{k + 1}^{(i)}$ and $\overline{v}_{k + 1} = \frac{1}{N}\sum_{i = 1}^N v_{k + 1}^{(i)}$.  $\rho_{max}(BB^T)$ is the maximum eigenvalues of the matrix $BB^T$.

\end{lemma}

\begin{proof}
Let $(x^*,v^*)$ be defined as in Lemma {\ref{lm1}}, and $(x_k^{(i)},v_k^{(i)})$ be the iterates of $i$'th client in Algorithm~\ref{FPDFP}. $T_{1,k}(\cdot) = \mathrm{Prox}_{\frac{\lambda}{\gamma_k}g^*}(\cdot)$,
and denote
\begin{equation*}\label{lm2eq1}
\begin{aligned}
\varphi_{1,k}^{(i)}(x,y) & = \tfrac{\lambda}{\gamma_k}B(x - \gamma_k \tilde{\nabla} f^{(i)}(x))+ (I - \lambda B B^T)y 
= \tfrac{\lambda}{\gamma_k}Bg_{1,k}^{(i)}(x) + My, \\
\textrm{and}~~~~\varphi_{2,k}(x,y) & = \tfrac{\lambda}{\gamma_k}B(x - \gamma_k\nabla f(x))+ (I - \lambda B B^T)y 
= \tfrac{\lambda}{\gamma_k}Bg_{2,k}(x) + My, \\
\end{aligned}
\end{equation*}
where $g_{1,k}^{(i)}(x) = x - \gamma_k \tilde{\nabla} f^{(i)}(x)$ and $g_{2,k}(x) = x - \gamma_k\nabla f(x)$, $M = I - \lambda B B^T$. \\
In the following, we derive the one step estimate for $x_k^{(i)}$ and $v_k^{(i)}$, respectively.
\begin{enumerate}[label=\textbf{\roman*}),leftmargin= *]
\item  Estimation of $\lVert v_{k + 1}^{(i)} - v^* \rVert_2^2$:
\begin{equation}\label{lm2eq2}
\begin{aligned}
\lVert v_{k + 1}^{(i)} - v^* \rVert_2^2
& = \lVert T_{1,k}(\varphi_{1,k}^{(i)}(x_k,v_k)) - v^* \rVert_2^2 \\
& = \lVert T_{1,k}(\varphi_{1,k}^{(i)}(x_k,v_k)) - T_{1,k}(\varphi_{2,k}(x^*,v^*)) \rVert_2^2 \\
& \leq \big< T_{1,k}(\varphi_{1,k}^{(i)}(x_k,v_k)) - T_{1,k}(\varphi_{2,k}(x^*,v^*)),\varphi_{1,k}^{(i)}(x_k,v_k) - \varphi_{2,k}(x^*,v^*) \big> \\
& = \frac{\lambda}{\gamma_k}\big< T_{1,k}(\varphi_{1,k}^{(i)}(x_k,v_k)) - T_{1,k}(\varphi_{2,k}(x^*,v^*)),B(g_{1,k}^{(i)}(x_k) - g_{2,k}(x^*)) \big> \\
& \qquad + \big< T_{1,k}(\varphi_{1,k}^{(i)}(x_k,v_k)) - T_{1,k}(\varphi_{2,k}(x^*,v^*)), M(v_k - v^*) \big>.
\end{aligned}
\end{equation}
The second equality follows from Eq. {(\ref{lm1eq1})} and the inequality follows from the firmly non-expansiveness of $T_1^{(k)}$.
For what follows, we denote $T_{1,k}(\varphi_{1,i}^{(k)}) = T_{1,k}(\varphi_{1,k}^{(i)}(x_k,v_k))$ and
$T_{1,k}(\varphi_2^{(k)}) = T_{1,k}(\varphi_2^{(k)}(x^*,v^*))$.

\item Estimation of $\lVert x_{k + 1}^{(i)} - x^* \rVert_2^2:$
\begin{equation}\label{lm2eq3}
\begin{aligned}
& \lVert x_{k + 1}^{(i)} - x^* \rVert_2^2  \\
& \overset{\tiny{\circled{1}}}{=} \lVert x_k - \gamma_k \tilde{\nabla} f^{(i)}(x_k) - \gamma_k B^T T_{1,k}(\varphi_{1,k}^{(i)}) - \big(x^* - \gamma_k \nabla f(x^*) - \gamma_k B^T T_{1,k}(\varphi_2^{(k)}) \big)  \rVert_2^2 \\
& = \lVert g_{1,k}^{(i)}(x_k) - \gamma_k B^T T_{1,k}(\varphi_{1,k}^{(i)}) - \big(g_{2,k}(x^*) - \gamma_k B^T T_{1,k}(\varphi_{2,k} \big) \rVert_2^2 \\
& = \lVert g_{1,k}^{(i)}(x_k) - g_{2,k}(x^*) - \gamma_k B^T \big(T_{1,k}(\varphi_{1,k}^{(i)}) - T_{1,k}(\varphi_{2,k}) \big)\rVert_2^2 \\
& = \lVert g_{1,k}^{(i)}(x_k)  - g_{2,k}(x^*) \rVert_2^2  - 2\gamma_k \big< B^T \big( T_{1,k}(\varphi_{1,k}^{(i)}) - T_1^{(k)}(\varphi_{2,k})\big), g_{1,k}^{(i)}(x_k) - g_{2,k}(x^*) \big> \\
&\qquad + \tfrac{\gamma_k^2}{\lambda^2}\lVert \lambda B^T ( T_1^{(k)}(\varphi_{1,i}^{(k)}) - T_{1,k}(\varphi_{2,k} ) \rVert_2^2 \\
& \overset{\tiny{\circled{2}}}{=}  \lVert g_{1,k}^{(i)}(x_k)  - g_{2,k}(x^*) \rVert_2^2  - 2\gamma_k \big< B^T \big( T_{1,k}(\varphi_{1,k}^{(i)}) - T_{1,k}(\varphi_{2,k})\big), g_{1,k}^{(i)}(x_k) - g_{2,k}(x^*) \big> \\
&\qquad - \tfrac{\gamma_k^2}{\lambda}\lVert ( T_{1,k}(\varphi_{1,k}^{(i)}) - T_{1,k}(\varphi_{2,k} ) \rVert_M^2 + \frac{\gamma_k^2}{\lambda}\lVert ( T_{1,k}(\varphi_{1,k}^{(i)}) - T_{1,k}(\varphi_{2,k}) ) \rVert_2^2.
\end{aligned}
\end{equation}
where
\begin{itemize}
  \item $\small{\circled{1}}$ follows from the update of $x_{k + 1}^{(i)}$ and {(\ref{lm1eq1})}.
  \item $\small{\circled{2}}$ we use the definition $M = I - \lambda BB^T$ and $\lVert y \rVert_M = \sqrt{\langle y,My \rangle}$.
\end{itemize}
\end{enumerate}
We are now sufficient to get the one step estimate of the Algorithm \ref{FPDFP} based on Lyapunov function 
$\lVert x_{k + 1}^{(i)} - x^* \rVert_2^2 + \frac{\gamma_{k + 1}^2}{\lambda}\lVert v_{k + 1}^{(i)} - v^* \rVert_2^2$ as follows: \\
\begin{equation}\label{lm2eq4}
\begin{aligned}
& \lVert x_{k + 1}^{(i)} - x^* \rVert_2^2 + \tfrac{\gamma_{k + 1}^2}{\lambda}\lVert v_{k + 1}^{(i)} - v^* \rVert_2^2 \\
& \overset{\tiny{\circled{1}}}{=} \lVert g_{1,k}^{(i)}(x_k)  - g_{2,k}(x^*) \rVert_2^2  - 2\gamma_k \big< B^T \big( T_{1,k}(\varphi_{1,k}^{(i)}) - T_{1,k}(\varphi_{2,k})\big), g_{1,k}^{(i)}(x_k) - g_{2,k}(x^*) \big> \\
&\qquad - \tfrac{\gamma_{k}^2}{\lambda}\lVert ( T_{1,k}(\varphi_{1,k}^{(i)}) - T_{1,k}(\varphi_{2,k}) ) \rVert_M^2 + \tfrac{\gamma_{k}^2}{\lambda}\lVert ( T_{1,k}(\varphi_{1,k}^{(i)}) - T_{1,k}(\varphi_{2,k}) ) \rVert_2^2 \\
&\qquad + \tfrac{\gamma_{k + 1}^2}{\lambda}\lVert ( T_{1,k}(\varphi_{1,k}^{(i)}) - T_{1,k}(\varphi_{2,k}) ) \rVert_2^2 \\
& \overset{\tiny{\circled{2}}}{\leq} \lVert g_{1,k}^{(i)}(x_k)  - g_{2,k}(x^*) \rVert_2^2  - 2\gamma_k \big< B^T \big( T_{1,k}(\varphi_{1,k}^{(i)}) - T_{1,k}(\varphi_{2,k})\big), g_{1,k}^{(i)}(x_k) - g_{2,k}(x^*) \big> \\
&\qquad - \tfrac{\gamma_{k}^2}{\lambda}\lVert ( T_{1,k}(\varphi_{1,k}^{(i)}) - T_{1,k}(\varphi_{2,k}) ) \rVert_M^2 + 2\tfrac{\gamma_{k}^2}{\lambda}\lVert ( T_{1,k}(\varphi_{1,k}^{(i)}) - T_{1,k}(\varphi_{2,k}) ) \rVert_2^2 \\
& \overset{\tiny{\circled{3}}}{\leq} \lVert g_{1,k}^{(i)}(x_k)  - g_{2,k}(x^*) \rVert_2^2  - \underline{2\gamma_k \big< B^T \big( T_{1,k}(\varphi_{1,k}^{(i)}) - T_{1,k}(\varphi_{2,k})\big), g_{1,k}^{(i)}(x_k) - g_{2,k}(x^*) \big>} \\
&\qquad - \tfrac{\gamma_{k}^2}{\lambda}\lVert ( T_{1,k}(\varphi_{1,k}^{(i)}) - T_{1,k}(\varphi_{2,k}) ) \rVert_M^2 
 + \underline{2\tfrac{\gamma_{k}^2}{\lambda} \tfrac{\lambda}{\gamma_k}\big< T_{1,k}(\varphi_{1,k}^{(i)}) - T_{1,k}(\varphi_2^{(k)}),B(g_{1,k}^{(i)}(x_k) - g_{2,k}(x^*)) \big>} \\
&\qquad + 2\tfrac{\gamma_{k}^2}{\lambda} \big< T_{1,k}(\varphi_{1,k}^{(i)}) - T_{1,k}(\varphi_{2,k}), M(v_k - v^*) \big>\\
& = \lVert g_{1,k}^{(i)}(x_k)  - g_{2,k}(x^*) \rVert_2^2 + 2\tfrac{\gamma_{k}^2}{\lambda}\big< T_{1,k}(\varphi_{1,k}^{(i)}) - T_{1,k}(\varphi_{2,k}), M(v_k - v^*) \big>  \\
& \qquad - \tfrac{\gamma_{k}^2}{\lambda}\lVert ( T_{1,k}(\varphi_{1,k}^{(i)}) - T_{1,k}(\varphi_{2,k}) ) \rVert_M^2 \\
& = \lVert g_{1,k}^{(i)}(x_k)  - g_{2,k}(x^*) \rVert_2^2 + \tfrac{\gamma_{k}^2}{\lambda}\lVert v_k - v^* \rVert_M^2 - \tfrac{\gamma_{k}^2}{\lambda}\lVert ( T_{1,k}(\varphi_{1,k}^{(i)}) - T_{1,k}(\varphi_{2,k}) ) - (v_k - v^*) \rVert_M^2 \\
& \overset{\tiny{\circled{4}}}{\leq} \lVert g_{1,k}^{(i)}(x_k)  - g_{2,k}(x^*) \rVert_2^2 + \tfrac{\gamma_{k}^2}{\lambda}\lVert v_k - v^* \rVert_M^2,
\end{aligned}
\end{equation}
where
\begin{itemize}
  \item $\small{\circled{1}}$ uses ({\ref{lm2eq3}}) .
  \item $\small{\circled{2}}$ follows from the fact that $\gamma_k$ is decreasing with respect to $k$.
  \item $\small{\circled{3}}$ uses ({\ref{lm2eq2}}).
  \item $\small{\circled{4}}$ uses the fact that $0 < \lambda \leq \frac{1}{\rho_{max}(B B^T)} $, which means that $M$ is positive semidefinite.
\end{itemize}
\noindent Taking the conditional expectation $\mathbb{E}_{R_{k}}^{(i)}[\cdot | \mathcal{F}_k]$ of both sides of {(\ref{lm2eq4})}, one obtains
\begin{equation}\label{lm2eq5}
\begin{aligned}
& \mathbb{E}_{R_{k}}^{(i)}\big[\lVert x_{k + 1}^{(i)} - x^* \rVert_2^2 + \tfrac{\gamma_{k+1}^2}{\lambda}\lVert v_{k + 1}^{(i)} - v^* \rVert_2^2 \big |\mathcal{F}_k \big]\\
 & \leq \mathbb{E}_{R_{k}}^{(i)}\big[\lVert g_{k,i}^{(1)}(x_k)  - g_k^{(2)}(x^*) \rVert_2^2 \big |\mathcal{F}_k \big]
+ \tfrac{\gamma_{k}^2}{\lambda}\lVert v_k - v^* \rVert_M^2  \\
& = \mathbb{E}_{R_{k}}^{(i)}\big[\lVert x_k - x^* - \gamma_k(\tilde{\nabla} f^{(i)}(x_k) - \nabla f(x^*)) \rVert_2^2 \big |\mathcal{F}_k \big] + \tfrac{\gamma_{k}^2}{\lambda}\lVert v_k - v^* \rVert_M^2 \\
& = \lVert x_k - x^* \rVert_2^2 + \tfrac{\gamma_{k}^2}{\lambda} \lVert v_k - v^* \rVert_M^2 
 - 2\gamma_k \big< \nabla f^{(i)}(x_k) - \nabla f(x^*), x_k - x^*\big> \\
&\qquad  + \gamma_k^2 \mathbb{E}_{R_{k}}^{(i)}[\lVert \tilde{\nabla} f^{(i)}(x_k) - \nabla f(x^*) \rVert_2^2 \big| \mathcal{F}_k],\\
\end{aligned}
\end{equation}
where, in the third term of the last equality, we use the fact that $\mathbb{E}_{R_{k}}^{(i)}[\tilde{\nabla} f^{(i)}(x_k) | \mathcal{F}_k) ] = \nabla f^{(i)}(x_k) $. 
Using the convexity of $\lVert \cdot \rVert_2^2$, we have 
\begin{equation}\label{lm2eq6}
\begin{aligned}
\lVert \overline{x}_{k + 1} - x^* \rVert_2^2 + \frac{\gamma_{k + 1}^2}{\lambda}\lVert \overline{v}_{k + 1}  - v^* \rVert_2^2
\leq \frac{1}{N}\sum_{i = 1}^N \lVert x_{k + 1}^{(i)} - x^* \rVert_2^2 + \frac{\gamma_{k + 1}^2}{\lambda}\lVert v_{k + 1}^{(i)}  - v^* \rVert_2^2. 
\end{aligned}
\end{equation}
Taking conditional expectation $\mathbb{E}_{R_{k}} = \prod_{i = 1}^N \mathbb{E}_{R_{k}}^{(i)}[\cdot/\mathcal{F}_k]$ on both side of {(\ref{lm2eq6})}, one obtains
\begin{equation}\label{lm2eq7}
\begin{aligned}
& \mathbb{E}_{R_{k}}[\lVert \overline{x}_{k + 1} - x^* \rVert_2^2 + \tfrac{\gamma_{k + 1}^2}{\lambda}\lVert \overline{v}_{k + 1}  - v^* \rVert_2^2|\mathcal{F}_k] \\
& \overset{\tiny{\circled{1}}}{\leq} \frac{1}{N}\prod_{i = 1}^N \mathbb{E}_{R_{k}}^{(i)} \big[ \sum_{i = 1}^N \lVert x_{k + 1}^{(i)} - x^* \rVert_2^2 + \tfrac{\gamma_{k + 1}^2}{\lambda}\lVert v_{k + 1}^{(i)}  - v^* \rVert_2^2 |\mathcal{F}_k\big] \\ 
& \overset{\tiny{\circled{2}}}{=} \frac{1}{N}\sum_{i = 1}^N \mathbb{E}_{R_{k}}^{(i)}[ \lVert x_{k + 1}^{(i)} - x^* \rVert_2^2 + \tfrac{\gamma_{k + 1}^2}{\lambda}\lVert v_{k + 1}^{(i)}  - v^* \rVert_2^2 | \mathcal{F}_k ]  \\
& \overset{\tiny{\circled{3}}}{\leq} \lVert x_k - x^* \rVert_2^2 + \tfrac{\gamma_{k}^2}{\lambda} \lVert v_k - v^* \rVert_M^2 
 - 2\gamma_k \big< \frac{1}{N} \sum_{i = 1}^N\nabla f^{(i)}(x_k) - \nabla f(x^*), x_k - x^*\big> \\
&\qquad  + \frac{\gamma_k^2}{N} \sum_{i = 1}^N\mathbb{E}_{R_{k}}^{(i)}[ \lVert \tilde{\nabla} f^{(i)}(x_k) - \nabla f(x^*) \rVert_2^2 \big| \mathcal{F}_k ] \\
& \overset{\tiny{\circled{4}}}{=} \lVert x_k - x^* \rVert_2^2 + \tfrac{\gamma_{k}^2}{\lambda} \lVert v_k - v^* \rVert_M^2 
 - 2\gamma_k \big< \nabla f(x_k) - \nabla f(x^*), x_k - x^*\big> \\
&\qquad  + \frac{\gamma_k^2}{N} \sum_{i = 1}^N\mathbb{E}_{R_{k }}^{(i)}[ \lVert \tilde{\nabla} f^{(i)}(x_k) - \nabla f(x^*) \rVert_2^2 \big| \mathcal{F}_k ], \\
\end{aligned}
\end{equation}
where
\begin{itemize}
  \item $\small{\circled{1}}$ uses  {(\ref{lm2eq6})}.
  \item $\small{\circled{2}}$ follows from the independence of sampling process of computing stochastic gradient across the clients.
  \item $\small{\circled{3}}$ uses {(\ref{lm2eq5})}.
  \item $\small{\circled{4}}$ uses the definition of $f$.
\end{itemize}
For each $i$
\begin{equation}\label{lm2eq8}
\begin{aligned}
& \mathbb{E}_{R_{k}}^{(i)}[ \lVert \tilde{\nabla} f^{(i)}(x_k) - \nabla f(x^*) \rVert_2^2 \big| \mathcal{F}_k ] \\
& = \mathbb{E}_{R_{k}}^{(i)}[ \lVert \tilde{\nabla} f^{(i)}(x_k) - \nabla f^{(i)}(x_k) 
+ \nabla f^{(i)}(x_k) - \nabla f^{(i)}(x^*)
+ \nabla f^{(i)}(x^*) - \nabla f(x^*) \rVert_2^2 \big| \mathcal{F}_k ] \\
& {\overset{\tiny{\circled{1}}}{\leq} 3\mathbb{E}_{R_{k}}^{(i)}[ \lVert \tilde{\nabla} f^{(i)}(x_k) - \nabla f^{(i)}(x_k) \rVert_2^2\big| \mathcal{F}_k ]
+ 3\mathbb{E}_{R_{k}}^{(i)}[ \lVert \nabla f^{(i)}(x_k) - \nabla f(x_k) \rVert_2^2 ]} \\
&{\qquad + 3\mathbb{E}_{R_{k}}^{(i)}[ \lVert \nabla f^{(i)}(x_k) - \nabla f(x^*) \rVert_2^2\big| \mathcal{F}_k ]} \\
& \overset{\tiny{\circled{2}}}{\leq} \tfrac{3}{\beta^2} \lVert x_k - x^*\rVert_2^2 + 3(\delta^2 + \sigma^2),
\end{aligned}  
\end{equation}
where
\begin{itemize}
  \item $\small{\circled{1}}$ follows from the inequality $\lVert a + b + c \rVert_2^2 \leq 3\lVert a \rVert_2^2 + 3\lVert b \rVert_2^2 + 3\lVert c \rVert_2^2$.  
  \item $\small{\circled{2}}$ uses the Assumption {\ref{smooth}, \ref{VR1}} and {\ref{VR2}}.
\end{itemize}
Combing {(\ref{lm2eq7})} and {(\ref{lm2eq8})}, we have
\begin{equation}
\begin{aligned}
& \mathbb{E}_{R_{k}}[ \lVert \overline{x}_{k + 1} - x^* \rVert_2^2 + \tfrac{\gamma_{k + 1}^2}{\lambda}\lVert \overline{v}_{k + 1}  - v^* \rVert_2^2 | \mathcal{F}_k ] \\
& \leq  (1 + \tfrac{3\gamma_k^2}{\beta^2} ) \lVert x_k - x^* \rVert_2^2 + \tfrac{\gamma_{k}^2}{\lambda} \lVert v_k - v^* \rVert_M^2 
 - 2\gamma_k \big< \nabla f(x_k) - \nabla f(x^*), x_k - x^*\big> + 3 \gamma_k^2(\delta^2 + \sigma^2).\\
\end{aligned}  
\end{equation}
This completes the proof. \qquad
\end{proof}

\subsection{Appendix B}
Observe that there are three sources of randomness in the update of FPDFP: (1) selection of clients; (2) computing of stochastic gradient (local update); (3) quantization. These three stochastic processes are independent. In this part, we focus on decoupling the three shadow sequences associated with the three stochastic processes (see Lemma {\ref{lm4}}) and provide estimates of their variance. 

\begin{lemma}\label{lm4}
If Assumption {\ref{Quantizer}} holds, we then have the following estimate 
\begin{equation}\label{lm4eq01}
\mathbb{E}_{k + 1}[ \lVert x_{k + 1} - x^*\rVert_2^2 ]  =  \mathbb{E}_{k + 1}[ \lVert x_{k + 1} - \hat{x}_{k + 1} \rVert_2^2 ]
 + \mathbb{E}_{k + 1}[ \lVert \hat{x}_{k + 1} - \overline{x}_{k + 1}\rVert_2^2 ]  
 + \mathbb{E}_{k + 1}[ \lVert \overline{x}_{k + 1} - x^*\rVert_2^2 ],  
\end{equation}
and 
\begin{equation}\label{lm4eq02}
\mathbb{E}_{k + 1}[ \lVert v_{k + 1} - v^*\rVert_2^2) ]
=  \mathbb{E}_{k + 1}[ \lVert v_{k + 1} - \hat{v}_{k + 1} \rVert_2^2 ]  
 + \mathbb{E}_{k + 1}[ \lVert \hat{v}_{k + 1} - \overline{v}_{k + 1}\rVert_2^2 ]  
 + \mathbb{E}_{k + 1}[ \lVert \overline{v}_{k + 1} - v^*\rVert_2^2],
\end{equation}
where $\mathbb{E}_{k + 1}[\cdot]$ is expectation of all randomness up to $k + 1$'th round.
$\hat{x}_{k + 1} = x_k + \frac{1}{N}\sum_{i = 1}^N Q(x_{k + 1}^{(i)} - x_k)$
and $\hat{v}_{k + 1} = \frac{1}{N}\sum_{i = 1}^N Q(v_{k + 1}^{(i)})$. The $\overline{x}_{k + 1}$ and $\overline{v}_{k + 1}$ are defined as that in Lemma \textbf{\ref{lm2}}. The $(x^*,v^*)$ are the optimal primal dual pair defined in Lemma \textbf{\ref{lm1}}.
\end{lemma}
\begin{proof}
First, we calculate the conditional expectation of $x_{k + 1}$ with respect to the random selection of clients as follows:
\begin{equation}\label{lm4eq1}
\begin{aligned}
\mathbb{E}_{\mathcal{S}_k}[ x_{k + 1} | (R_{k},Q,\mathcal{F}_k) ]
& = x_k + \mathbb{E}_{\mathcal{S}_k}\big[ \frac{1}{n}\sum_{i \in \mathcal{S}_k}Q(x_{k + 1}^{(i)} - x_k)| (R_{k},Q,\mathcal{F}_k) \big] \\
& = x_k + \frac{1}{\binom{N}{n}}\frac{1}{n}\binom{N - 1}{n - 1}\sum_{i = 1}^NQ(x_{k + 1}^{(i)} - x_k) \\
& = x_k + \frac{1}{N}\sum_{i = 1}^NQ(x_{k + 1}^{(i)} - x_k)\\
& = \hat{x}_{k + 1},
\end{aligned} 
\end{equation}
where $\binom{N}{n}$ is binomial coefficient defined by $\frac{N(N - 1)\cdots(N - n + 1)}{n(n - 1)\cdots 1}$. The expression $\mathbb{E}_{\mathcal{S}_k}[ \cdot | (R_{k},Q,\mathcal{F}_k)]$ is the expectation for selecting clients conditioned on local stochastic gradient update, quantization and $\mathcal{F}_k$. 
The equation {(\ref{lm4eq1})} states that $x_{k + 1}$ is an unbiased estimate of $\hat{x}_{k + 1}$. \\
Furthermore, by using the unbiasedness of quantization, one gets
\begin{equation}\label{lm4eq2}
\begin{aligned}
\mathbb{E}_{Q}[ \hat{x}_{k + 1}| (R_{k},\mathcal{F}_k) ]
& = x_k + \mathbb{E}_{Q}\big[ \frac{1}{N}\sum_{i = 1}Q(x_{k + 1}^{(i)} - x_k) \big| (R_{k},\mathcal{F}_k) \big] 
= \frac{1}{N}\sum_{i = 1}x_{k + 1}^{(i)} 
 = \overline{x}_{k + 1}.
\end{aligned} 
\end{equation}
Here $\mathbb{E}_{Q}[ \cdot| (R_{k},\mathcal{F}_k) ]$ is the expectation with respect to quantization conditioned on local stochastic gradient update and $\mathcal{F}_k$. \\ 
Combine {(\ref{lm4eq1})} and {(\ref{lm4eq2})}, we have
\begin{equation}\label{lm4eq3}
\begin{aligned}
& \mathbb{E}_{Q}\mathbb{E}_{\mathcal{S}_k}[ \lVert x_{k + 1} - x^*\rVert_2^2 | (R_{k},\mathcal{F}_k) ] \\ 
 & = \mathbb{E}_{Q}\mathbb{E}_{\mathcal{S}_k}[ \lVert x_{k + 1} - \hat{x}_{k + 1} + \hat{x}_{k + 1} - \overline{x}_{k + 1} + \overline{x}_{k + 1} - x^*\rVert_2^2 | (R_{k},\mathcal{F}_k)  ] \\
 & =  \mathbb{E}_{Q}\mathbb{E}_{\mathcal{S}_k}[ \lVert x_{k + 1} - \hat{x}_{k + 1} \rVert_2^2 | (R_{k},\mathcal{F}_k)  ] 
 + \mathbb{E}_{Q}\mathbb{E}_{\mathcal{S}_k}[ \lVert \hat{x}_{k + 1} - \overline{x}_{k + 1}\rVert_2^2 |(R_{k},\mathcal{F}_k)  ] \\
 & \qquad + \mathbb{E}_{Q}\mathbb{E}_{\mathcal{S}_k}[ \lVert \overline{x}_{k + 1} - x^*\rVert_2^2 | (R_{k},\mathcal{F}_k)  ]  + 2\mathbb{E}_{Q}\mathbb{E}_{\mathcal{S}_k}[ \langle x_{k + 1} - \hat{x}_{k + 1}, \hat{x}_{k + 1} - \overline{x}_{k + 1}\rangle | (R_{k},\mathcal{F}_k)  ] \\
 & \qquad + 2\mathbb{E}_{Q}\mathbb{E}_{\mathcal{S}_k}[ \langle x_{k + 1} - \hat{x}_{k + 1}, \overline{x}_{k + 1} - x^*\rangle | (R_{k},\mathcal{F}_k) ] + 2\mathbb{E}_{Q}\mathbb{E}_{\mathcal{S}_k}[ \langle \hat{x}_{k + 1} - \overline{x}_{k + 1}, \overline{x}_{k + 1} - x^*\rangle | (R_{k},\mathcal{F}_k) ] \\
& = \mathbb{E}_{Q}\mathbb{E}_{\mathcal{S}_k}[ \lVert x_{k + 1} - \hat{x}_{k + 1} \rVert_2^2 |(R_{k},\mathcal{F}_k) ]
 + \mathbb{E}_{Q}\mathbb{E}_{\mathcal{S}_k}[ \lVert \hat{x}_{k + 1} - \overline{x}_{k + 1}\rVert_2^2 | (R_{k},\mathcal{F}_k)  ]  + \mathbb{E}_{Q}\mathbb{E}_{\mathcal{S}_k}[ \lVert \overline{x}_{k + 1} - x^*\rVert_2^2 | (R_{k},\mathcal{F}_k)  ]  \\
& \qquad  + 2\mathbb{E}_{Q}( \mathbb{E}_{\mathcal{S}_k}[ \langle \underline{x_{k + 1} - \hat{x}_{k + 1}}, \hat{x}_{k + 1} - \overline{x}_{k + 1}\rangle | (R_{k},Q,\mathcal{F}_k))| (R_{k},\mathcal{F}_k)  ] \\
 & \qquad + 2\mathbb{E}_{Q}(\mathbb{E}_{\mathcal{S}_k}[ \langle \underline{x_{k + 1} - \hat{x}_{k + 1}}, \overline{x}_{k + 1} - x^*\rangle | (R_{k},Q,\mathcal{F}_k))| (R_{k},\mathcal{F}_k) ] \\
 & \qquad + 2\mathbb{E}_{\mathcal{S}_k}\mathbb{E}_{Q}[ \langle \underline{\hat{x}_{k + 1} - \overline{x}_{k + 1}}, \overline{x}_{k + 1} - x^*\rangle | (R_{k},\mathcal{F}_k)] \\
& =  \mathbb{E}_{Q}\mathbb{E}_{\mathcal{S}_k}[ \lVert x_{k + 1} - \hat{x}_{k + 1} \rVert_2^2 | (R_{k},\mathcal{F}_k) ]
 + \mathbb{E}_{Q}\mathbb{E}_{\mathcal{S}_k}[ \lVert \hat{x}_{k + 1} - \overline{x}_{k + 1}\rVert_2^2 | (R_{k},\mathcal{F}_k) ] \\
& \qquad + \mathbb{E}_{Q}\mathbb{E}_{\mathcal{S}_k}[ \lVert \overline{x}_{k + 1} - x^*\rVert_2^2 | (R_{k},\mathcal{F}_k) ],
 \end{aligned}
\end{equation}
where the third equality uses the tower property of conditional expectation and independence of quantization and random selection of clients. \\
Taking expectation $\mathbb{E}_{R_{k}}[\cdot]$ of both sides of {(\ref{lm4eq3})},  
\begin{equation}\label{lm4eq31}
\begin{aligned}
& \mathbb{E}_{k + 1}[ \lVert x_{k + 1} - x^*\rVert_2^2 | \mathcal{F}_{k}] \\
& = \mathbb{E}_{R_{k}}[ \mathbb{E}_{Q}\mathbb{E}_{\mathcal{S}_k}[\lVert x_{k + 1} - x^*\rVert_2^2 | \mathcal{F}_{k} ]] \\
 & =  \mathbb{E}_{k + 1}[\lVert x_{k + 1} - \hat{x}_{k + 1} \rVert_2^2 | \mathcal{F}_{k}] 
 + \mathbb{E}_{k + 1}[\lVert \hat{x}_{k + 1} - \overline{x}_{k + 1}\rVert_2^2 | \mathcal{F}_{k}]  
 + \mathbb{E}_{k + 1}[\lVert \overline{x}_{k + 1} - x^*\rVert_2^2 | \mathcal{F}_{k}] , 
 \end{aligned}
\end{equation} 
where the above equalities follow from the fact $\mathbb{E}_{k + 1}[\cdot| \mathcal{F}_{k}] = \mathbb{E}_{R_{k}}\mathbb{E}_{Q}\mathbb{E}_{\mathcal{S}_k}[ \cdot | \mathcal{F}_{k} ]$. Then the Eq. {(\ref{lm4eq01})} follows by taking expectation $\mathbb{E}_{k}(\cdot)$ of both sides of Eq. {(\ref{lm4eq31})}. \\
Analogously, we can get  
\begin{equation}\label{lm4eq4}
\begin{aligned}
\mathbb{E}_{\mathcal{S}_k}[ v_{k + 1}| (R_{k},Q,\mathcal{F}_k) ]
& = \mathbb{E}_{\mathcal{S}_k}\big[ \frac{1}{n}\sum_{i \in \mathcal{S}_k}Q(v_k^{(i)}) \big| (R_{k},Q,\mathcal{F}_k) \big] 
 = \frac{1}{N}\sum_{i = 1}^NQ(v_k^{(i)}) 
  = \hat{v}_{k + 1},
\end{aligned} 
\end{equation}
and 
\begin{equation}\label{lm4eq5}
\begin{aligned}
\mathbb{E}_{Q}[ \hat{v}_{k + 1} | (R_{k},\mathcal{F}_k) ]
& = \mathbb{E}_{Q}\big[ \frac{1}{N}\sum_{i = 1}Q(v_k^{(i)}) | (R_{k},\mathcal{F}_k) \big] 
 = \frac{1}{N}\sum_{i = 1}v_{k + 1}^{(i)} 
 = \overline{v}_{k + 1},
\end{aligned} 
\end{equation}
which results in
\begin{equation}\label{lm4eq6}
\begin{aligned}
& \mathbb{E}_{Q}\mathbb{E}_{\mathcal{S}_k}[ \lVert v_{k + 1} - v^*\rVert_2^2 | (R_{k},\mathcal{F}_k) ] \\
& = \mathbb{E}_{Q}\mathbb{E}_{\mathcal{S}_k}[ \lVert v_{k + 1} - \hat{v}_{k + 1} + \hat{v}_{k + 1} - \overline{v}_{k + 1} + \overline{v}_{k + 1} - v^*\rVert_2^2 | (R_{k},\mathcal{F}_k) ] \\
& = \mathbb{E}_{Q}\mathbb{E}_{\mathcal{S}_k}[ \lVert v_{k + 1} - \hat{v}_{k + 1} \rVert_2^2 | (R_{k},\mathcal{F}_k) ]
 + \mathbb{E}_{Q}\mathbb{E}_{\mathcal{S}_k}[ \lVert \hat{v}_{k + 1} - \overline{v}_{k + 1}\rVert_2^2 | (R_{k},\mathcal{F}_k) ] + \mathbb{E}_{Q}\mathbb{E}_{\mathcal{S}_k}[ \lVert \overline{v}_{k + 1} - v^*\rVert_2^2 | (R_{k},\mathcal{F}_k) ] . 
 \end{aligned}
\end{equation} 
Again the Eq. {(\ref{lm4eq02})} follows by taking expectation $\mathbb{E}_{R_{k}}[\cdot]$ and $\mathbb{E}_{k}[\cdot]$ successively of both sides of {(\ref{lm4eq6})}.
\end{proof}

\begin{lemma}\label{lm5}
Suppose Assumptions {\ref{smooth}}, {\ref{Quantizer}} and {\ref{fact1}} hold, then the following inequalities hold
\begin{equation}\label{lm5eq01}
 \mathbb{E}_{k + 1}[ \lVert x_{k + 1} - \hat{x}_{k + 1}  \rVert_2^2 + \tfrac{\gamma_{k + 1}^2}{\lambda} \lVert v_{k + 1} - \hat{v}_{k + 1}  \rVert_2^2 ]  \\
\leq \gamma_k^2C_1 \mathbb{E}_{k}[ \lVert x_k - x^*\rVert_2^2 ] + \gamma_k ^2C_3,
\end{equation}
and 
\begin{equation}\label{lm5eq02}
 \mathbb{E}_{k + 1}[ \lVert \lVert \hat{x}_{k + 1} - \overline{x}_{k + 1}   \rVert_2^2 + \tfrac{\gamma_{k + 1}^2}{\lambda}\lVert \hat{v}_{k + 1} - \overline{v}_{k + 1}  \rVert_2^2 ]  \\
\leq \gamma_k^2C_2 \mathbb{E}_{k}[ \lVert x_k - x^*\rVert_2^2 ] + \gamma_k ^2C_4,
\end{equation}
where $C_1 = \frac{32(1 + q)(N - n)}{n(N - 1)\beta^2}$, $C_2 = \frac{8(1 + q)}{\beta^2}$. Let $C_0 = 4\sigma^2 + 8\delta^2 + 2\rho_{\max}(BB^T)(2M^2+ \frac{M^2}{\lambda})$ and the constant $M$ is the upper bound of the dual variable in Assumption \ref{fact1}, then 
$C_3 = \frac{4(1 + q)(N - n)}{n(N - 1)}C_0$, $C_4 = (1 + q)C_0$. 

\end{lemma}

\begin{proof}
{We first prove {(\ref{lm5eq01})}.} 
Let $z_{k + 1}^{(i)} = Q(x_{k + 1}^{(i)} - x_k)$ and $\overline{z}_{k + 1} = \frac{1}{N}\sum_{i = 1}^Nz_{k + 1}^{(i)}$, then
\begin{equation}\label{lm5eq1}
\begin{aligned}
& \mathbb{E}_{\mathcal{S}_k}[ \lVert x_{k + 1} - \hat{x}_{k + 1}  \rVert_2^2 | (R_{k},Q,\mathcal{F}_k) ] \\
& = \mathbb{E}_{\mathcal{S}_k}\big[ \big\lVert \frac{1}{n}\sum_{i \in \mathcal{S}_k} z_{k + 1}^{(i)} - \overline{z}_{k + 1}  \big\Vert_2^2  \big| (R_{k},Q,\mathcal{F}_k) \big] \\
& =  \frac{1}{n^2}\mathbb{E}_{\mathcal{S}_k}\big[ \big\lVert \sum_{i = 1}^N \mathbb{I}(i \in \mathcal{S}_k)(z_{k + 1}^{(i)} - \overline{z}_{k + 1} )\big\rVert_2^2  \big| (R_{k},Q,\mathcal{F}_k)   \big] \\
& =  \frac{1}{n^2}\big \{ \sum_{i = 1}^N \mathrm{Pr}(i \in \mathcal{S}_k) \lVert z_{k + 1}^{(i)} - \overline{z}_{k + 1} \rVert_2^2 \}
+ \sum_{i \not = j}\mathrm{Pr}(i,j\in \mathcal{S}_k) \langle z_{k + 1}^{(i)} - \overline{z}_{k + 1},z_{k + 1}^{(j)} - \overline{z}_{k + 1} \rangle) \\
& =  \frac{1}{nN} \sum_{i = 1}^N \lVert z_{k + 1}^{(i)} - \overline{z}_{k + 1} \rVert_2^2 
 + \frac{n - 1}{nN(N - 1)}\sum_{i \not = j}\langle z_{k + 1}^{(i)} - \overline{z}_{k + 1},z_{k + 1}^{(j)} - \overline{z}_{k + 1} \rangle) \\
& =  \frac{N - n}{nN(N - 1)} \sum_{i = 1}^N \lVert z_{k + 1}^{(i)} - \overline{z}_{k + 1} \rVert_2^2, \\
\end{aligned}
\end{equation}
where $\mathbb{I}(\cdot)$ is the indicator function that equals $1$ when the event occurs and $0$ otherwise. The last equality follows from the fact that $\sum_{i = 1}^N \lVert z_{k + 1}^{(i)} - \overline{z}_{k + 1} \rVert_2^2 + \sum_{i \not = j}\langle z_{k + 1}^{(i)} - \overline{z}_{k + 1},z_{k + 1}^{(j)} - \overline{z}_{k + 1} \rangle) = 0$. 
Taking conditional expectation $\mathbb{E}_{Q}[\cdot | (R_{k},\mathcal{F}_k) ]$ on the summation of the last equality of {(\ref{lm5eq1})}, we have
\begin{equation}\label{lm5eq2}
\begin{aligned}
\sum_{i = 1}^N \mathbb{E}_Q [ \lVert z_{k + 1}^{(i)} - \overline{z}_{k + 1} \rVert_2^2 | (R_{k},\mathcal{F}_k) ] 
& \overset{\tiny{\circled{1}}}{\leq} 2\sum_{i  = 1}^N\mathbb{E}_Q [ \lVert z_{k + 1}^{(i)} \rVert_2^2 | (R_{k},\mathcal{F}_k)] + 2N \mathbb{E}_Q[\lVert \overline{z}_{k + 1} \rVert_2^2 | (R_{k},\mathcal{F}_k) ] \\
& \overset{\tiny{\circled{2}}}{\leq} 4\sum_{i  = 1}^N\mathbb{E}_Q[ \lVert z_{k + 1}^{(i)} \rVert_2^2 | (R_{k},\mathcal{F}_k) ] \\
& = 4\sum_{i  = 1}^N\mathbb{E}_Q[\lVert Q(x_{k + 1}^{(i)} - x_k) \rVert_2^2 | (R_{k},\mathcal{F}_k) ] \\
& \overset{\tiny{\circled{3}}}{\leq} 4(1 + q)\sum_{i  = 1}^N\lVert x_{k + 1}^{(i)} - x_k \rVert_2^2, 
\end{aligned}
\end{equation}
where 
\begin{itemize}
  \item $\circled{1}$ follows from the inequality $\lVert a + b \rVert_2^2 \leq 2\lVert a \rVert_2^2 + 2\lVert b \rVert_2^2$ for given $a,b$.
  \item $\circled{2}$ uses the convexity of $\lVert \cdot \rVert_2^2$.
  \item $\circled{3}$ uses the fact that, under Assumption {\ref{Quantizer}}, $\mathbb{E}[\lVert Q_s(x)\rVert_2^2 | x] \leq (1 + q) \lVert x \rVert_2^2$. 
\end{itemize}
Combing {(\ref{lm5eq1})} and {(\ref{lm5eq2})}, we finally get  
\begin{equation}\label{lm5eq3}
\begin{aligned}
& \mathbb{E}_{Q}\mathbb{E}_{\mathcal{S}_k}[ \lVert x_{k + 1} - \hat{x}_{k + 1} \rVert_2^2 | (R_{k},\mathcal{F}_k) ]  \\
& = \mathbb{E}_{Q}(\mathbb{E}_{\mathcal{S}_k}[ \lVert x_{k + 1} - \hat{x}_{k + 1} \rVert_2^2 | (R_{k},Q,\mathcal{F}_k)) |(R_{k},\mathcal{F}_k)  ] \\
& \leq \frac{4(1 + q)(N - n)}{nN(N - 1)}\sum_{i  = 1}^N\lVert x_{k + 1}^{(i)} - x_k \rVert_2^2. \\
\end{aligned}
\end{equation}
Similarly, we have 
  \begin{equation}\label{lm5eq4}
  \begin{aligned}
  \mathbb{E}_{\mathcal{S}_k}[ \lVert v_{k + 1} - \hat{v}_{k + 1} \rVert_2^2 | (R_{k},Q,\mathcal{F}_k) ]
  & =  \frac{N - n}{nN(N - 1)} \sum_{i = 1}^N \lVert Q(v_{k + 1}^{(i)}) - \hat{v}_{k + 1} \rVert_2^2, 
  \end{aligned}
  \end{equation}
and 
  \begin{equation}\label{lm5eq5}
  \begin{aligned}
  &\sum_{i = 1}^N \mathbb{E}_Q [\lVert Q(v_{k + 1}^{(i)}) - \hat{v}_{k + 1} \rVert_2^2 | (R_{k},\mathcal{F}_k) ] \\
  & \leq 2\sum_{i  = 1}^N\mathbb{E}_Q [ \lVert Q( v_{k + 1}^{(i)} \rVert_2^2) | (R_{k},\mathcal{F}_k)] + 2N \mathbb{E}_Q(\lVert \hat{v}_{k + 1} \rVert_2^2| (R_{k},\mathcal{F}_k) ] \\
  & \leq 4\sum_{i  = 1}^N\mathbb{E}_Q [ \lVert Q(v_{k + 1}^{(i)}) \rVert_2^2 | (R_{k},\mathcal{F}_k) ] \\
  & \leq 4(1 + q)\sum_{i  = 1}^N\lVert v_{k + 1}^{(i)}\rVert_2^2.
  \end{aligned}
  \end{equation}
  Combing {(\ref{lm5eq4})} and {(\ref{lm5eq5})}, we have 
  \begin{equation}\label{lm5eq6}
  \begin{aligned}
  &\mathbb{E}_{Q}\mathbb{E}_{\mathcal{S}_k}[ \lVert v_{k + 1} - \hat{v}_{k + 1} \rVert_2^2 | (R_{k},\mathcal{F}_k) ] \\
  & = \mathbb{E}_{Q}[\mathbb{E}_{\mathcal{S}_k}[ \lVert v_{k + 1} - \hat{v}_{k + 1} \rVert_2^2 | (R_{k},Q,\mathcal{F}_k)]| (R_{k},\mathcal{F}_k) ] \\ 
  & \leq \frac{4(1 + q)(N - n)}{nN(N - 1)}\sum_{i  = 1}^N\lVert v_{k + 1}^{(i)}\rVert_2^2.
  \end{aligned}
  \end{equation}
  Combing {(\ref{lm5eq3})} and {(\ref{lm5eq6})} yields
  \begin{equation}\label{lm5eq61}
  \begin{aligned}
  & \mathbb{E}_{Q}\mathbb{E}_{\mathcal{S}_k}[ \lVert x_{k + 1} - \hat{x}_{k + 1} \rVert_2^2 + \frac{\gamma_{k + 1}^2}{\lambda}\lVert v_{k + 1} - \hat{v}_{k + 1} \rVert_2^2 | (R_{k},\mathcal{F}_k) ] \\
  & \leq \frac{4(1 + q)(N - n)}{nN(N - 1)}\sum_{i  = 1}^N\lVert x_{k + 1}^{(i)} - x_k \rVert_2^2 + \frac{\gamma_{k + 1}^2}{\lambda}\lVert v_{k + 1}^{(i)}\rVert_2^2.  
  \end{aligned}
  \end{equation}
We are now bounding the right hand side of {(\ref{lm5eq61})}.  \\
For each $i \in \{1,2,\cdots,N\}$
\begin{equation}\label{lm5eq7}
\begin{aligned}
& \mathbb{E}_{R_{k}^{(i)}}[ \lVert x_{k + 1}^{(i)} - x_k \rVert_2^2 | \mathcal{F}_k ] \\
& = \mathbb{E}_{R_{k}^{(i)}}[ \gamma_k ^2 \lVert \tilde{\nabla} f^{(i)}(x_k) - B^Tv_{k + 1}^{(i)}  \rVert_2^2 | \mathcal{F}_k ] \\
& = \gamma_k ^2\mathbb{E}_{R_{k}^{(i)}}[ \lVert \tilde{\nabla} f^{(i)}(x_k) - \nabla f(x^*) + (B^T v_{k + 1}^{(i)} - B^T v^*)  \rVert_2^2 | \mathcal{F}_k ] \\
& \overset{\tiny{\circled{1}}}{\leq} 2\gamma_k ^2\mathbb{E}_{R_{k}^{(i)}}[ \lVert \tilde{\nabla} f^{(i)}(x_k) - \nabla f(x^*)\rVert_2^2 | \mathcal{F}_k )+ 2\gamma_k^2\mathbb{E}_{R_{k}^{(i)}}(\lVert B^T v_{k + 1}^{(i)} - B^T v^* \rVert_2^2 | \mathcal{F}_k] \\
& \overset{\tiny{\circled{2}}}{\leq}  4\gamma_k^2 \mathbb{E}_{R_{k}^{(i)}}[\lVert \tilde{\nabla} f^{(i)}(x_k) - \nabla f^{(i)}(x_k) \rVert_2^2 | \mathcal{F}_k] + 4\gamma_k^2\mathbb{E}_{R_{k}^{(i)}}[\lVert \nabla f^{(i)}(x_k)- \nabla f(x^*)\rVert_2^2| \mathcal{F}_k ] \\
& \qquad + 2\gamma_k^2\rho_{\max}(BB^T) \mathbb{E}_{R_{k}^{(i)}}[ \lVert v_{k + 1}^{(i)} - v^* \rVert_2^2| \mathcal{F}_k ] \\
& \overset{\tiny{\circled{3}}}{\leq}  4\gamma_k^2 \mathbb{E}_{R_{k}^{(i)}}[\lVert \tilde{\nabla} f^{(i)}(x_k) - \nabla f^{(i)}(x_k) \rVert_2^2 | \mathcal{F}_k] + 8\gamma_k^2\mathbb{E}_{R_{k}^{(i)}}[\lVert \nabla f^{(i)}(x_k) - \nabla f(x_k)\rVert_2^2| \mathcal{F}_k] \\
&\qquad + 8\gamma_k^2\mathbb{E}_{R_{k}^{(i)}}[\lVert \nabla f(x_k)- \nabla f(x^*)\rVert_2^2| \mathcal{F}_k] + 2\gamma_k^2\rho_{\max}(BB^T) \mathbb{E}_{R_{k}^{(i)}}[\lVert v_{k + 1}^{(i)} - v^* \rVert_2^2| \mathcal{F}_k] \\
& \overset{\tiny{\circled{4}}}{\leq} 4\gamma_k ^2 \sigma^2 + 8\gamma_k ^2 \delta^2 + 8\gamma_k ^2\lVert \nabla f(x_k)- \nabla f(x^*)\rVert_2^2 + 2\gamma_k^2\rho_{\max}(BB^T)\mathbb{E}_{R_{k}^{(i)}}[\lVert v_{k + 1}^{(i)} - v^* \rVert_2^2 | \mathcal{F}_k]\\
& \overset{\tiny{\circled{5}}}{\leq}  4\gamma_k^2 \sigma^2 + 8\gamma_k ^2 \delta^2 + 8\frac{\gamma_k^2}{\beta^2}\lVert x_k - x^*\rVert_2^2 + 2\gamma_k^2\rho_{\max}(BB^T)\mathbb{E}_{R_{k}^{(i)}}[ \lVert v_{k + 1}^{(i)} - v^* \rVert_2^2| \mathcal{F}_k ],  
\end{aligned} 
\end{equation}
where 
\begin{itemize}
  \item $\circled{1}, \circled{2}, \circled{3}$ uses the inequality $\lVert a + b \rVert_2^2 \leq 2\lVert a \rVert_2^2 + 2\lVert b \rVert_2^2$ for given $a,b$.
  \item $\circled{4}$ uses Assumption {\ref{VR1}} and Assumption {\ref{VR2}}.
  \item $\circled{5}$ uses Assumption {\ref{smooth}}.
\end{itemize}
We then have 
\begin{equation}\label{lm5eq8}
\begin{aligned}
& \mathbb{E}_{R_{k}^{(i)}}[ \lVert x_{k + 1}^{(i)} - x_k \rVert_2^2  +  \tfrac{\gamma_{k + 1}^2}{\lambda}\lVert v_{k + 1}^{(i)}\rVert_2^2 | \mathcal{F}_k ] \\
& \overset{\tiny{\circled{1}}}{\leq} 8\frac{\gamma_k ^2}{\beta^2}\lVert x_k - x^*\rVert_2^2 + \gamma_k ^2 (4\sigma^2 + 8\delta^2 + 2\rho_{\max}(BB^T)\mathbb{E}_{R_{k}^{(i)}}[ \lVert v_{k + 1}^{(i)} - v^* \rVert_2^2 + \tfrac{\lVert v_{k + 1}^{(i)}\rVert_2^2}{\lambda} | \mathcal{F}_k ] \\
& \overset{\tiny{\circled{2}}}{\leq} 8\frac{\gamma_k ^2}{\beta^2}\lVert x_k - x^*\rVert_2^2 + \gamma_k ^2 (4\sigma^2 + 8\delta^2 + 2\rho_{\max}(BB^T)(2M^2+ \frac{M^2}{\lambda}))  \\
& \overset{\tiny{\circled{3}}}{=} 8\frac{\gamma_k ^2}{\beta^2}\lVert x_k - x^*\rVert_2^2 + \gamma_k ^2 C_0,\\
\end{aligned}
\end{equation}
where 
\begin{itemize}
  \item $\circled{1}$ follows from the Eq. {(\ref{lm5eq7})} and the decreasing property of $\gamma_k$.
  \item $\circled{2}$ use the boundness of $\lVert v_{k + 1}^{(i)}\rVert_2^2$ in Assumption \ref{fact1}.
  \item $\circled{3}$ use the definition $C_0 = 4\sigma^2 + 8\delta^2 + 2\rho_{\max}(BB^T)(2M^2+ \frac{M^2}{\lambda})$.
\end{itemize}
Combine {(\ref{lm5eq3})}, {(\ref{lm5eq6})} and {(\ref{lm5eq8})}, we then have  
\begin{equation}\label{lm5eq81}
\begin{aligned}
 & \mathbb{E}_{k + 1}[ \lVert x_{k + 1} - \hat{x}_{k + 1}  \rVert_2^2 + \tfrac{\gamma_{k + 1}^2}{\lambda}\lVert v_{k + 1} - \hat{v}_{k + 1}  \rVert_2^2 ] \\
 & \overset{\tiny{\circled{1}}}{=}  \mathbb{E}_{k}\Big[\mathbb{E}_{R_{k}}\big[ \mathbb{E}_{Q}\mathbb{E}_{\mathcal{S}_k}[\lVert x_{k + 1} - \hat{x}_{k + 1}  \rVert_2^2 + \tfrac{\gamma_{k + 1}^2}{\lambda}\lVert v_{k + 1} - \hat{v}_{k + 1}  \rVert_2^2 | (\mathcal{R}_k,\mathcal{F}_k) ] | \mathcal{F}_k \big]\Big]  \\
&  \overset{\tiny{\circled{2}}}{\leq}  \frac{4(1 + q)(N - n)}{nN(N - 1)}\mathbb{E}_{k}\Big[\mathbb{E}_{R_{k}} \big[ \sum_{i = 1}^N \lVert x_{k + 1}^{(i)} - x_k \rVert_2^2  +  \tfrac{\gamma_{k + 1}^2}{\lambda}\lVert v_{k + 1}^{(i)}\rVert_2^2 | \mathcal{F}_k \big]\Big] \\
& \overset{\tiny{\circled{3}}}{=}  \frac{4(1 + q)(N - n)}{nN(N - 1)}\mathbb{E}_{k}\Big[\sum_{i = 1}^N \mathbb{E}_{R_{k}}^{(i)}  \big[\lVert x_{k + 1}^{(i)} - x_k \rVert_2^2  +  \tfrac{\gamma_{k + 1}^2}{\lambda}\lVert v_{k + 1}^{(i)}\rVert_2^2 \big | \mathcal{F}_k \big]\Big] \\
& \overset{\tiny{\circled{4}}}{\leq}  \frac{4(1 + q)(N - n)}{n(N - 1)}(8\tfrac{\gamma_k ^2}{\beta^2}\mathbb{E}_{k}[ \lVert x_k - x^*\rVert_2^2)] + \gamma_k ^2C_0)  \\
& \overset{\tiny{\circled{5}}}{=}  \gamma_k^2C_1\mathbb{E}_{k}[ \lVert x_k - x^*\rVert_2^2 ] + \gamma_k ^2C_3,
\end{aligned}
\end{equation}
where
\begin{itemize}
  \item $\circled{1}$ follows tower property.
  \item $\circled{2}$ uses Eq. {(\ref{lm5eq61})}.
  \item $\circled{3}$ uses the definition of $\mathbb{E}_{R_{k}} $ and the independence of stochastic gradient updates across the clients.
  \item $\circled{4}$ uses Eq. {(\ref{lm5eq8})}.
  \item $\circled{5}$ uses the definition of $C_1$ and $C_3$.
\end{itemize}

\vskip1mm

Analogously, the {(\ref{lm5eq02})} can be proved as follows:
  \begin{equation}\label{lm5eq9}
  \begin{aligned}
  \mathbb{E}_{Q}[ \lVert \hat{x}_{k + 1} - \overline{x}_{k + 1} \rVert_2^2 | (R_{k},\mathcal{F}_k) ] 
  & = \mathbb{E}_{Q} \big[ \big \lVert x_k + \frac{1}{N}\sum_{i = 1}^N Q(x_{k + 1}^{(i)} - x_k) - \frac{1}{N}\sum_{i = 1}^N x_{k + 1}^{(i)} \big \rVert_2^2 \big | (R_{k},\mathcal{F}_k) \big]  \\
  & = \mathbb{E}_{Q} \big[ \big \lVert \frac{1}{N}\sum_{i = 1}^N  Q(x_{k + 1}^{(i)} - x_k) - (x_{k + 1}^{(i)} - x_k) \big \rVert_2^2 \big| (R_{k},\mathcal{F}_k) \big] \\
  & \leq \frac{1}{N} \mathbb{E}_{Q} \big[ \sum_{i = 1}^N \lVert Q(x_{k + 1}^{(i)} - x_k) - (x_{k + 1}^{(i)} - x_k) \rVert_2^2 \big| (R_{k},\mathcal{F}_k)\big] \\
  & \leq \frac{1}{N} \sum_{i = 1}^N \mathbb{E}_{Q} \big[\lVert Q(x_{k + 1}^{(i)} - x_k) \rVert_2^2 \big| (R_{k},\mathcal{F}_k)\big]  \\
  & \leq \frac{1 + q}{N} \sum_{i = 1}^N \lVert x_{k + 1}^{(i)} - x_k \rVert_2^2,
  \end{aligned}
  \end{equation}
where the last two inequalities follows from Assumption {\ref{Quantizer}}.  \\
And $\mathbb{E}_{Q}[\lVert \hat{v}_{k + 1} - \overline{v}_{k + 1} \rVert_2^2 | (R_{k},\mathcal{F}_k) ]$ is bounded by 
  \begin{equation}\label{lm5eq10}
  \begin{aligned}
  &\mathbb{E}_{Q}[\lVert \hat{v}_{k + 1} - \overline{v}_{k + 1} \rVert_2^2 | (R_{k},\mathcal{F}_k)) ] \\
  & = \mathbb{E}_{Q} \big[ \big \lVert \frac{1}{N}\sum_{i = 1}^N Q(v_{k + 1}^{(i)}) - \frac{1}{N}\sum_{i = 1}^N v_{k + 1}^{(i)} \big \rVert_2^2 \big| (R_{k},\mathcal{F}_k)) \big]  \\
  & = \mathbb{E}_{Q} \big[ \big \lVert \frac{1}{N}\sum_{i = 1}^N Q(v_{k + 1}^{(i)}) - v_{k + 1}^{(i)} \big \rVert_2^2 | (R_{k},\mathcal{F}_k)) \big]  \\
  & \leq  \frac{1}{N}\sum_{i = 1}^N  \mathbb{E}_{Q} \big[ \big \lVert Q(v_{k + 1}^{(i)}) - v_{k + 1}^{(i)} \big \rVert_2^2 \big| (R_{k},\mathcal{F}_k)) \big] \\
  & \leq \frac{1 + q}{N} \sum_{i = 1}^N \lVert v_{k + 1}^{(i)} \rVert_2^2.   
  \end{aligned}
  \end{equation}
Then 
\begin{equation}\label{lm5eq11}
\begin{aligned}
& \mathbb{E}_{k + 1}[ \lVert \hat{x}_{k + 1} - \overline{x}_{k + 1}  \rVert_2^2 + \tfrac{\gamma_{k + 1}^2}{\lambda}\lVert \hat{v}_{k + 1} - \overline{v}_{k + 1}  \rVert_2^2 ]  \\
 & = \mathbb{E}_{k}\Big[\mathbb{E}_{R_{k}}\big[ \mathbb{E}_{Q}[\lVert \hat{x}_{k + 1} - \overline{x}_{k + 1}  \rVert_2^2 + \tfrac{\gamma_{k + 1}^2}{\lambda}\lVert \hat{v}_{k + 1} - \overline{v}_{k + 1}  \rVert_2^2 | (\mathcal{R}_k,\mathcal{F}_k) ] | \mathcal{F}_k \big] \Big]  \\
& \overset{\tiny{\circled{1}}}{\leq} \frac{1 + q}{N}\mathbb{E}_{k}\Big[ \mathbb{E}_{R_{k}} \big[ \sum_{i = 1}^N\lVert x_{k + 1}^{(i)} - x_k \rVert_2^2  +  \tfrac{\gamma_{k + 1}^2}{\lambda}\lVert v_{k + 1}^{(i)}\rVert_2^2 | \mathcal{F}_k \big]\Big] \\
& = \frac{1 + q}{N}\mathbb{E}_{k}\Big[ \sum_{i = 1}^N\mathbb{E}_{R_{k}}^{(i)} \big [\lVert x_{k + 1}^{(i)} - x_k \rVert_2^2)  +  \tfrac{\gamma_{k + 1}^2}{\lambda}\lVert v_{k + 1}^{(i)}\rVert_2^2 \big | \mathcal{F}_k \big] \Big] \\
& \overset{\tiny{\circled{2}}}{\leq} (1 + q)(8\tfrac{\gamma_k ^2}{\beta^2}\mathbb{E}_{k}[ \lVert x_k - x^*\rVert_2^2) + \gamma_k ^2C_0 ] \\
& = \frac{8(1 + q)\gamma_k ^2}{\beta^2}\mathbb{E}_{k}[ \lVert x_k - x^*\rVert_2^2 ] + \gamma_k ^2(1 + q)C_0 \\
& \overset{\tiny{\circled{3}}}{=} C_2\gamma_k^2\mathbb{E}_{k}[ \lVert x_k - x^*\rVert_2^2 ] + \gamma_k ^2C_4,
\end{aligned}
\end{equation}
where 
\begin{itemize}
   \item $\circled{1}$ follows from {(\ref{lm5eq9})} and {(\ref{lm5eq10})}.
   \item $\circled{2}$ uses {(\ref{lm5eq8})}.
   \item $\circled{2}$ uses the definition of $C_2$ and $C_4$.
 \end{itemize} 
This completes the proof. \\ 
\end{proof}

\subsection{Appendix C}
We are now proving the Lemma \ref{lm6}
\begin{proof}[{Proof of Lemma {\ref{lm6}}}]
Recall Eq. {(\ref{lm2eq0})},  we have
\begin{equation}\label{lm6eq1}
\begin{aligned}
& \mathbb{E}_{R_k}[ \lVert \overline{x}_{k + 1} - x^* \rVert_2^2 + \tfrac{\gamma_{k + 1}^2}{\lambda}\lVert \overline{v}_{k + 1}  - v^* \rVert_2^2/\mathcal{F}_k ] \\
& \leq  (1 + \tfrac{3\gamma_k^2}{\beta^2})\lVert x_k - x^* \rVert_2^2 + \tfrac{\gamma_{k}^2}{\lambda} \lVert v_k - v^* \rVert_M^2
 - 2\gamma_k \big< \nabla f(x_k) - \nabla f(x^*), x_k - x^*\big> + \gamma_k^2(\delta^2 + \sigma^2)\\
 & \leq  (1 + \tfrac{3\gamma_k^2}{\beta^2}\lVert x_k - x^* \rVert_2^2 + \tfrac{\gamma_{k}^2}{\lambda}(1 -  \lambda\rho_{min}(BB^T)) \lVert v_k - v^* \rVert_2^2
 - 2\gamma_k \big< \nabla f(x_k) - \nabla f(x^*), x_k - x^*\big> \\
 & \qquad + \gamma_k^2(\delta^2 + \sigma^2)\\
  & \leq  (1 - 2\gamma_k \mu + \tfrac{3\gamma_k^2}{\beta^2})
  \lVert x_k - x^* \rVert_2^2 + \tfrac{\gamma_{k}^2}{\lambda}(1 - \lambda\rho_{min}(BB^T)) \lVert v_k - v^* \rVert_2^2+ \gamma_k^2(\delta^2 + \sigma^2), \\
\end{aligned}  
\end{equation}
where the last inequality uses the strongly convexity of $f(x)$. The $\rho_{min}(BB^T))$ denotes the minimum eigenvalue of matrix $BB^T$. \\
Taking expectation $\mathbb{E}_{k}[\cdot]$ on both sides of {(\ref{lm6eq1})} yields
 \begin{equation}\label{lm6eq2}
\begin{aligned}
& \mathbb{E}_{k + 1}[\lVert \overline{x}_{k + 1} - x^* \rVert_2^2 + \tfrac{\gamma_{k + 1}^2}{\lambda}\lVert \overline{v}_{k + 1}  - v^* \rVert_2^2 ] \\
& = \mathbb{E}_{k}[\lVert \overline{x}_{k + 1} - x^* \rVert_2^2 + \tfrac{\gamma_{k + 1}^2}{\lambda}\lVert \overline{v}_{k + 1}  - v^* \rVert_2^2 ] \\
& \leq (1 - 2\gamma_k \mu + \tfrac{3\gamma_k^2}{\beta^2})\mathbb{E}_{k}[ \lVert x_k - x^* \rVert_2^2 ] + \tfrac{\gamma_{k}^2}{\lambda}(1 -  \lambda\rho_{min}(BB^T)) \mathbb{E}_{k}[ \lVert v_k - v^* \rVert_2^2 ] + \gamma_k^2(\delta^2 + \sigma^2),\\
\end{aligned}  
\end{equation}
where the first equality follows from the fact that the randomness of partial participation and quantization dose not exits in  $\overline{x}_{k + 1}$ and $\overline{v}_{k + 1}$. \\
Using Lemma {\ref{lm4}}, one has 
\begin{equation}\label{lm6eq3}
\begin{aligned}
 &\mathbb{E}_{k + 1}[ \lVert x_{k + 1} - x^*\rVert_2^2 + \tfrac{\gamma_{k + 1}^2}{\lambda}\lVert v_{k + 1} - v^*\rVert_2^2 ] \\
 & =  \mathbb{E}_{k + 1}[ \lVert x_{k + 1} - \hat{x}_{k + 1} \rVert_2^2 + \tfrac{\gamma_{k + 1}^2}{\lambda}\lVert v_{k + 1} - \hat{v}_{k + 1} \rVert_2^2 ]  
 + \mathbb{E}_{k + 1}[ \lVert \hat{x}_{k + 1} - \overline{x}_{k + 1}\rVert_2^2 + \tfrac{\gamma_{k + 1}^2}{\lambda}\lVert \hat{v}_{k + 1} - \overline{v}_{k + 1} \rVert_2^2 ]  \\
& \qquad + \mathbb{E}_{k + 1}[ \lVert \overline{x}_{k + 1} - x^*\rVert_2^2 + \tfrac{\gamma_{k + 1}^2}{\lambda}\lVert \overline{v}_{k + 1} - v^* \rVert_2^2], 
 \end{aligned}
\end{equation}
Apply Lemma {\ref{lm5}} and {(\ref{lm6eq2})} to {(\ref{lm6eq3})}, we finally arrive at 
\begin{equation}\label{lm6eq4}
\begin{aligned}
 &\mathbb{E}_{k + 1}[ \lVert x_{k + 1} - x^*\rVert_2^2 + \tfrac{\gamma_{k + 1}^2}{\lambda}\lVert v_{k + 1} - v^*\rVert_2^2 ] \\
 & =  \mathbb{E}_{k + 1}[ \lVert x_{k + 1} - \hat{x}_{k + 1} \rVert_2^2 + \tfrac{\gamma_{k + 1}^2}{\lambda}\lVert v_{k + 1} - \hat{v}_{k + 1} \rVert_2^2 ] 
 + \mathbb{E}_{k + 1}[ \lVert \hat{x}_{k + 1} - \overline{x}_{k + 1}\rVert_2^2 + \tfrac{\gamma_{k + 1}^2}{\lambda}\lVert \hat{v}_{k + 1} - \overline{v}_{k + 1} \rVert_2^2 ]  \\
& \qquad + \mathbb{E}_{k + 1}[ \lVert \overline{x}_{k + 1} - x^*\rVert_2^2 + \tfrac{\gamma_{k + 1}^2}{\lambda}\lVert \overline{v}_{k + 1} - v^* \rVert_2^2 ]  \\
& \leq (1 - 2\gamma_k \mu + (\tfrac{3}{\beta^2} + C_1 + C_2)\gamma_k^2))\mathbb{E}_{k}[ \lVert x_k - x^* \rVert_2^2 ] + \tfrac{\gamma_{k}^2}{\lambda}[1 -  \lambda\rho_{min}(BB^T)] \mathbb{E}_{k}[ \lVert v_k - v^* \rVert_2^2 ] \\
& \quad  + \gamma_k^2(\delta^2 + \sigma^2 + C_3 + C_4)\\
& \overset{\tiny{\circled{1}}}{=}  (1 - \gamma_k(2\mu - (\tfrac{3}{\beta^2} + C_1 + C_2)\gamma_k))\mathbb{E}_{k}[ \lVert x_k - x^* \rVert_2^2 ] + \tfrac{\gamma_{k}^2}{\lambda}(1 -  \lambda\rho_{min}(BB^T)) \mathbb{E}_{k}[\lVert v_k - v^* \rVert_2^2 ] + \gamma_k^2D_2\\
& \overset{\tiny{\circled{2}}}{\leq}  (1 - D_1\gamma_k)\mathbb{E}_{k}[ \lVert x_k - x^* \rVert_2^2 ] + \tfrac{\gamma_{k}^2}{\lambda}(1 -  \lambda \rho_{min}(BB^T)) \mathbb{E}_{k}[\lVert v_k - v^* \rVert_2^2 ] + \gamma_k^2D_2,
 \end{aligned}
\end{equation}
where 
\begin{itemize}
   \item In $\circled{1}$, we set $D_2 = \delta^2 + \sigma^2 + C_3 + C_4 = \sigma^2 + \delta^2 
+ \Big(\frac{4(1 + q)(N - n)}{n(N - 1)} + 1 + q\Big)C_0$.
   \item In $\circled{2}$, we use the fact that $\gamma_k$ is decreasing and $0 < D_1 < 2\mu$, there must be some $K_0$ larger enough such that $\gamma_k(2\mu - (\tfrac{3}{\beta^2} + C_1 + C_2)\gamma_k) > D_1\gamma_k$.
 \end{itemize} 
 This completes the proof.
\end{proof}


\bibliographystyle{siam}
\bibliography{ref}

\begin{thebibliography}{10}

\bibitem{QSGD}
{\sc D.~Alistarh, D.~Grubic, J.~Li, R.~Tomioka, and M.~Vojnovic}, {\em Qsgd:
  Communication-efficient sgd via gradient quantization and encoding}, Advances
  in neural information processing systems, 30 (2017).

\bibitem{multirecon1}
{\sc R.~Ammanouil, A.~Ferrari, R.~Flamary, C.~Ferrari, and D.~Mary}, {\em
  Multi-frequency image reconstruction for radio-interferometry with self-tuned
  regularization parameters}, in 2017 25th European Signal Processing
  Conference (EUSIPCO), IEEE, 2017, pp.~1435--1439.

\bibitem{HE}
{\sc Y.~Aono, T.~Hayashi, L.~Wang, S.~Moriai, et~al.}, {\em Privacy-preserving
  deep learning via additively homomorphic encryption}, IEEE Transactions on
  Information Forensics and Security, 13 (2017), pp.~1333--1345.

\bibitem{Backdoor}
{\sc E.~Bagdasaryan, A.~Veit, Y.~Hua, D.~Estrin, and V.~Shmatikov}, {\em How to
  backdoor federated learning}, in International Conference on Artificial
  Intelligence and Statistics, PMLR, 2020, pp.~2938--2948.

\bibitem{GLasso}
{\sc O.~Banerjee, L.~El~Ghaoui, and A.~d'Aspremont}, {\em Model selection
  through sparse maximum likelihood estimation for multivariate gaussian or
  binary data}, The Journal of Machine Learning Research, 9 (2008),
  pp.~485--516.

\bibitem{baofast}
{\sc Y.~Bao, M.~Crawshaw, S.~Luo, and M.~Liu}, {\em Fast composite optimization
  and statistical recovery in federated learning}, in International Conference
  on Machine Learning, PMLR, 2022, pp.~1508--1536.

\bibitem{cvxbook}
{\sc H.~H. Bauschke, P.~L. Combettes, et~al.}, {\em Convex analysis and
  monotone operator theory in Hilbert spaces}, vol.~408, Springer, 2011.

\bibitem{signSGD}
{\sc J.~Bernstein, Y.-X. Wang, K.~Azizzadenesheli, and A.~Anandkumar}, {\em
  signsgd: Compressed optimisation for non-convex problems}, in International
  Conference on Machine Learning, PMLR, 2018, pp.~560--569.

\bibitem{PSG}
{\sc K.~Bonawitz, V.~Ivanov, B.~Kreuter, A.~Marcedone, H.~B. McMahan, S.~Patel,
  D.~Ramage, A.~Segal, and K.~Seth}, {\em Practical secure aggregation for
  federated learning on user-held data}, arXiv preprint arXiv:1611.04482,
  (2016).

\bibitem{SGD}
{\sc L.~Bottou}, {\em Stochastic learning}, in Summer School on Machine
  Learning, Springer, 2003, pp.~146--168.

\bibitem{TVPGD}
{\sc A.~Chambolle and P.-L. Lions}, {\em Image recovery via total variation
  minimization and related problems}, Numerische Mathematik, 76 (1997),
  pp.~167--188.

\bibitem{CP}
{\sc A.~Chambolle and T.~Pock}, {\em A first-order primal-dual algorithm for
  convex problems with applications to imaging}, Journal of mathematical
  imaging and vision, 40 (2011), pp.~120--145.

\bibitem{LIBSVM}
{\sc C.-C. Chang and C.-J. Lin}, {\em Libsvm: a library for support vector
  machines}, ACM transactions on intelligent systems and technology (TIST), 2
  (2011), pp.~1--27.

\bibitem{FedSGD}
{\sc J.~Chen, X.~Pan, R.~Monga, S.~Bengio, and R.~Jozefowicz}, {\em Revisiting
  distributed synchronous sgd}, arXiv preprint arXiv:1604.00981,  (2016).

\bibitem{PDFP}
{\sc P.~Chen, J.~Huang, and X.~Zhang}, {\em A primal--dual fixed point
  algorithm for convex separable minimization with applications to image
  restoration}, Inverse Problems, 29 (2013), p.~025011.

\bibitem{com}
{\sc P.~L. Combettes and J.-C. Pesquet}, {\em Primal-dual splitting algorithm
  for solving inclusions with mixtures of composite, lipschitzian, and
  parallel-sum type monotone operators}, Set-Valued and variational analysis,
  20 (2012), pp.~307--330.

\bibitem{PGD}
{\sc P.~L. Combettes and V.~R. Wajs}, {\em Signal recovery by proximal
  forward-backward splitting}, Multiscale modeling \& simulation, 4 (2005),
  pp.~1168--1200.

\bibitem{condat}
{\sc L.~Condat}, {\em A primal--dual splitting method for convex optimization
  involving lipschitzian, proximable and linear composite terms}, Journal of
  optimization theory and applications, 158 (2013), pp.~460--479.

\bibitem{PAPC}
{\sc Y.~Drori, S.~Sabach, and M.~Teboulle}, {\em A simple algorithm for a class
  of nonsmooth convex--concave saddle-point problems}, Operations Research
  Letters, 43 (2015), pp.~209--214.

\bibitem{CMI}
{\sc J.~C. Duchi, S.~Shalev-Shwartz, Y.~Singer, and A.~Tewari}, {\em Composite
  objective mirror descent.}, in COLT, vol.~10, Citeseer, 2010, pp.~14--26.

\bibitem{Elias}
{\sc P.~Elias}, {\em Universal codeword sets and representations of the
  integers}, IEEE transactions on information theory, 21 (1975), pp.~194--203.

\bibitem{plogistic}
{\sc J.~Friedman, T.~Hastie, and R.~Tibshirani}, {\em Regularization paths for
  generalized linear models via coordinate descent}, Journal of statistical
  software, 33 (2010), p.~1.

\bibitem{Xray}
{\sc H.~Gao}, {\em Fast parallel algorithms for the x-ray transform and its
  adjoint}, Medical physics, 39 (2012), pp.~7110--7120.

\bibitem{DP}
{\sc R.~C. Geyer, T.~Klein, and M.~Nabi}, {\em Differentially private federated
  learning: A client level perspective}, arXiv preprint arXiv:1712.07557,
  (2017).

\bibitem{MARINA}
{\sc E.~Gorbunov, K.~P. Burlachenko, Z.~Li, and P.~Richt{\'a}rik}, {\em Marina:
  Faster non-convex distributed learning with compression}, in International
  Conference on Machine Learning, PMLR, 2021, pp.~3788--3798.

\bibitem{HeYuan}
{\sc B.~He and X.~Yuan}, {\em Convergence analysis of primal-dual algorithms
  for a saddle-point problem: from contraction perspective}, SIAM Journal on
  Imaging Sciences, 5 (2012), pp.~119--149.

\bibitem{GANlk}
{\sc B.~Hitaj, G.~Ateniese, and F.~Perez-Cruz}, {\em Deep models under the gan:
  information leakage from collaborative deep learning}, in Proceedings of the
  2017 ACM SIGSAC conference on computer and communications security, 2017,
  pp.~603--618.

\bibitem{PM}
{\sc S.~Kakade, O.~Shamir, K.~Sindharan, and A.~Tewari}, {\em Learning
  exponential families in high-dimensions: Strong convexity and sparsity}, in
  Proceedings of the thirteenth international conference on artificial
  intelligence and statistics, JMLR Workshop and Conference Proceedings, 2010,
  pp.~381--388.

\bibitem{DCGD}
{\sc S.~Khirirat, H.~R. Feyzmahdavian, and M.~Johansson}, {\em Distributed
  learning with compressed gradients}, arXiv preprint arXiv:1806.06573,
  (2018).

\bibitem{FL1}
{\sc J.~Kone{\v{c}}n{\`y}, H.~B. McMahan, D.~Ramage, and P.~Richt{\'a}rik},
  {\em Federated optimization: Distributed machine learning for on-device
  intelligence}, arXiv preprint arXiv:1610.02527,  (2016).

\bibitem{multirecon2}
{\sc S.~Lan, Z.~Wang, A.~K. Roy-Chowdhury, E.~Wei, and Q.~Zhu}, {\em
  Distributed multi-agent video fast-forwarding}, in Proceedings of the 28th
  ACM International Conference on Multimedia, 2020, pp.~1075--1084.

\bibitem{Fedprox}
{\sc T.~Li, A.~K. Sahu, M.~Zaheer, M.~Sanjabi, A.~Talwalkar, and V.~Smith},
  {\em Federated optimization in heterogeneous networks}, Proceedings of
  Machine Learning and Systems, 2 (2020), pp.~429--450.

\bibitem{ADIANA}
{\sc Z.~Li, D.~Kovalev, X.~Qian, and P.~Richt{\'a}rik}, {\em Acceleration for
  compressed gradient descent in distributed and federated optimization}, arXiv
  preprint arXiv:2002.11364,  (2020).

\bibitem{multirecon3}
{\sc F.-H. Lin, K.~K. Kwong, J.~W. Belliveau, and L.~L. Wald}, {\em Parallel
  imaging reconstruction using automatic regularization}, Magnetic Resonance in
  Medicine: An Official Journal of the International Society for Magnetic
  Resonance in Medicine, 51 (2004), pp.~559--567.

\bibitem{LV}
{\sc I.~Loris and C.~Verhoeven}, {\em On a generalization of the iterative
  soft-thresholding algorithm for the case of non-separable penalty}, Inverse
  Problems, 27 (2011), p.~125007.

\bibitem{FedAvg}
{\sc B.~McMahan, E.~Moore, D.~Ramage, S.~Hampson, and B.~A. y~Arcas}, {\em
  Communication-efficient learning of deep networks from decentralized data},
  in Artificial intelligence and statistics, PMLR, 2017, pp.~1273--1282.

\bibitem{pDLG}
{\sc L.~Melis, C.~Song, E.~De~Cristofaro, and V.~Shmatikov}, {\em Exploiting
  unintended feature leakage in collaborative learning}, in 2019 IEEE Symposium
  on Security and Privacy (SP), IEEE, 2019, pp.~691--706.

\bibitem{DIANA}
{\sc K.~Mishchenko, E.~Gorbunov, M.~Tak{\'a}{\v{c}}, and P.~Richt{\'a}rik},
  {\em Distributed learning with compressed gradient differences}, arXiv
  preprint arXiv:1901.09269,  (2019).

\bibitem{IntSGD}
{\sc K.~Mishchenko, B.~Wang, D.~Kovalev, and P.~Richt{\'a}rik}, {\em Intsgd:
  Floatless compression of stochastic gradients}, arXiv preprint
  arXiv:2102.08374,  (2021).

\bibitem{SADMM}
{\sc H.~Ouyang, N.~He, L.~Tran, and A.~Gray}, {\em Stochastic alternating
  direction method of multipliers}, in International conference on machine
  learning, PMLR, 2013, pp.~80--88.

\bibitem{NUQSGD}
{\sc A.~Ramezani-Kebrya, F.~Faghri, I.~Markov, V.~Aksenov, D.~Alistarh, and
  D.~M. Roy}, {\em Nuqsgd: Provably communication-efficient data-parallel sgd
  via nonuniform quantization.}, J. Mach. Learn. Res., 22 (2021), pp.~114--1.

\bibitem{Fedpaq}
{\sc A.~Reisizadeh, A.~Mokhtari, H.~Hassani, A.~Jadbabaie, and R.~Pedarsani},
  {\em Fedpaq: A communication-efficient federated learning method with
  periodic averaging and quantization}, in International Conference on
  Artificial Intelligence and Statistics, PMLR, 2020, pp.~2021--2031.

\bibitem{PSGD}
{\sc L.~Rosasco, S.~Villa, and B.~C. V{\~u}}, {\em Convergence of stochastic
  proximal gradient algorithm}, Applied Mathematics \& Optimization, 82 (2020),
  pp.~891--917.

\bibitem{TV}
{\sc L.~I. Rudin, S.~Osher, and E.~Fatemi}, {\em Nonlinear total variation
  based noise removal algorithms}, Physica D: nonlinear phenomena, 60 (1992),
  pp.~259--268.

\bibitem{Smith}
{\sc V.~Smith, C.-K. Chiang, M.~Sanjabi, and A.~S. Talwalkar}, {\em Federated
  multi-task learning}, Advances in neural information processing systems, 30
  (2017).

\bibitem{LASSO}
{\sc R.~Tibshirani}, {\em Regression shrinkage and selection via the lasso},
  Journal of the Royal Statistical Society: Series B (Methodological), 58
  (1996), pp.~267--288.

\bibitem{PowerSGD}
{\sc T.~Vogels, S.~P. Karimireddy, and M.~Jaggi}, {\em Powersgd: Practical
  low-rank gradient compression for distributed optimization}, Advances in
  Neural Information Processing Systems, 32 (2019).

\bibitem{vu}
{\sc B.~C. V{\~u}}, {\em A splitting algorithm for dual monotone inclusions
  involving cocoercive operators}, Advances in Computational Mathematics, 38
  (2013), pp.~667--681.

\bibitem{FCM}
{\sc H.~Yuan, M.~Zaheer, and S.~Reddi}, {\em Federated composite optimization},
  in International Conference on Machine Learning, PMLR, 2021,
  pp.~12253--12266.

\bibitem{FedPD}
{\sc X.~Zhang, M.~Hong, S.~Dhople, W.~Yin, and Y.~Liu}, {\em Fedpd: A federated
  learning framework with optimal rates and adaptivity to non-iid data}, arXiv
  preprint arXiv:2005.11418,  (2020).

\bibitem{SPDFP}
{\sc Y.-N. Zhu and X.~Zhang}, {\em Stochastic primal dual fixed point method
  for composite optimization}, Journal of Scientific Computing, 84 (2020),
  pp.~1--25.

\end{thebibliography}
\end{document}